\documentclass[12pt]{amsart}
\usepackage{amsfonts, amssymb}
\usepackage{amsmath}
\usepackage{amsthm}
\usepackage{cite}
\numberwithin{equation}{section}

\newtheorem{theorem}{Theorem}[section]
\newtheorem{lemma}[theorem]{Lemma}

\newtheorem{corollary}[theorem]{Corollary}

\newtheorem{proposition}[theorem]{Proposition}

\begin{document}

\title[Non-oscillatory solutions]{On non-oscillatory solutions of the \\
generalized Hill equation}

\author{Yueyang Zhang}
%\footnote{*Corresponding author:~zhangyueyang@ustb.edu.cn}
%\footnote{Corresponding author}
%\authornote{*Corresponding author}
\address{School of Mathematics and Physics, University of Science and Technology Beijing, No.~30 Xueyuan Road, Haidian, Beijing, 100083, P.R. China}
\email{zhangyueyang@ustb.edu.cn}

%\author{Yinjie Feng}
%\address{School of Mathematics and Physics, University of Science and Technology Beijing, No.~30 Xueyuan Road, Haidian, Beijing, 100083, P.R. China}
%\email{m202410716@ustb.edu.cn}

\thanks{The author is supported by grants from a Project supported by the National Natural Science Foundation of China~{(12301091)}}



\date{\today}

\commby{}

\begin{abstract}
Let $\varrho\neq -1$ be a complex number with modulus $|\varrho|=1$ and $K(x,y)=P_1(x)+P_2(y)$ for two polynomials $P_1$ and $P_2$. We consider the non-oscillatory solutions such that $\lambda(f)<\infty$ of the generalized Hill equation $f''-K(e^z,e^{-\varrho z})f=0$ ($\sharp$). When $\varrho=1$, the Hill equation $f''-K(e^z,e^{-z})f=0$ ($\dag$) is also written as $x^2u''-[K(x,x^{-1})-1/4]u=0$ ($\ddag$). We point out that there is a full correspondence between the non-oscillatory solutions of equation ($\dag$) and the Liouvillian solutions of equation ($\ddag$). Then this paper has two purposes. First, we show that if equation ($\sharp$) has a nonzero non-oscillatory solution, then $\varrho=1$. To this end, we solve entire solutions of a general Tumura--Clunie type differential equation. Second, for the particular Hill equation $f''-(e^{\mathbf{k}z}+b_{\mathbf{s}}e^{\mathbf{s}z}+b_0)f=0$, where $\mathbf{k}>\mathbf{s}\geq 1$ are integers and $b_{\mathbf{s}}\not=0$, we use Kovacic's algorithms to determine the non-oscillatory solutions with relatively few zeros.
\end{abstract}

\maketitle

%\newpage

\section{Introduction}\label{se1: introduction}

The value distribution theory in complex differential equations has been a subject in the last several decades; see \cite{Laine1993} and references therein. Many important results are obtained under the framework of the Nevanlinna theory; see, e.g., \cite{Hayman1964Meromorphic,Laine1993} for the standard notation and basic results. Bank and Laine~\cite{Banklaine1982,Banklaine1982-2} initiated the study on the complex oscillation of the second order linear differential equation
\begin{equation}\label{bank-laine Eq}
f''+Af=0,
\end{equation}
where $A$ is a transcendental entire function. All nonzero solutions of equation \eqref{bank-laine Eq} are entire functions having order of growth $\sigma(f)=\infty$. Bank and Laine~\cite{Banklaine1982,Banklaine1982-2} proved: For two linearly independent solutions $f_1$ and $f_2$ of equation \eqref{bank-laine Eq}, if $\sigma(A)\not\in \mathbb{N}$, then $\lambda(f_1f_2)\geq \sigma(A)$; if $\sigma(A)<1/2$, then $\lambda(f_1f_2)=\infty$. The condition $\sigma(A)<1/2$ was relaxed to $\sigma(A)=1/2$ by Shen~\cite{Shen1985} and Rossi~\cite{Rossi1986} independently. \emph{The Bank--Laine conjecture} asserts that $\lambda(f_1f_2)=\infty$ whenever $\sigma(A)\not\in \mathbb{N}$. This conjecture motivated a lot of research in the last several decades; see the surveys~\cite{Gundersen2014,lainetohge2008} and references therein. In recent years, Bergweiler and Eremenko~\cite{Bergweilereremenko2017,Bergweilereremenko2019} disproved the Bank--Laine conjecture and the present author~\cite{Zhang2024,Zhang2025} fulfilled their construction of Bank--Laine functions.

We say that $f$ is a \emph{non-oscillatory} solution of equation \eqref{bank-laine Eq} if $\lambda(f)<\infty$. Let $\varrho \neq -1$ be a complex number with modulus $|\varrho|=1$ and $K(x,y)=P_1(x)+P_2(y)$ for two polynomials $P_1$ and $P_2$ of the form
\begin{equation}\label{bank-laine0-fu-0}
K(x,y)=b_{\mathbf{k}}x^{\mathbf{k}}+\cdots+b_{1}x+b_0+b_{-1}y+\cdots+b_{-\mathbf{l}}y^{\mathbf{l}},
\end{equation}
where $\mathbf{k}\geq 1$, $\mathbf{l}\geq 0$ are two integers, $b_{\mathbf{k}}$, $\cdots$, $b_{-\mathbf{l}}$ are constants such that $b_{\mathbf{k}}\not=0$. Then $K(e^{z},e^{-\varrho z})$ is an entire function having order of growth~$1$. In this paper, we are concerned with the question if the \emph{generalized Hill equation}
\begin{equation}\label{Hill Eq General}
f''-K(e^{z},e^{-\varrho z})f=0
\end{equation}
can have nonzero non-oscillatory solutions. When $\varrho=1$, $K(x,x^{-1})$ is rational in $x$ and analytic in $\mathbb{C}-\{0\}$ and $K(e^{z},e^{-z})$ is an entire periodic function of period $2\pi\mathbf{i}$. Here and in the following, we use the bolded $\mathbf{i}$ to denote the \emph{imaginary unit}. A remarkable result in \cite{Banklaine1983-1,Chiang2000} states that, if the Hill equation
\begin{equation}\label{Hill Eq}
f''-K(e^{z},e^{-z})f=0
\end{equation}
has a nonzero non-oscillatory solution $f$, then there exist complex constants $c$, a polynomial $\psi$ with simple roots only and a Laurent polynomial $\chi$ such that
\begin{equation}\label{Bank-laine0-eq-1}
f(z)=\psi(e^{z/\mathbf{p}})e^{cz}e^{\chi(e^{z/\mathbf{p}})},
\end{equation}
where $\mathbf{p}=1$ if $\mathbf{k}$ and $\mathbf{l}$ are both even, or $\mathbf{p}=2$ if $\mathbf{k}$ is an odd positive integer and $b_{-1}=\cdots=b_{-\mathbf{l}}=0$. We may write the Laurent polynomial $\chi$ in the form $\chi(z)=\sum_{i=0}^{\mathbf{k}/2-1}c_{i}z^{\mathbf{k}/2-i}+\sum_{j=0}^{\mathbf{l}/2-1}d_{j}z^{-(\mathbf{l}/2-j)}$ in the first case and $\chi(z)=\sum_{i=0}^{\mathbf{k}-1}c_{i}z^{\mathbf{k}-i}$ such that $c_{i}=0$ whenever $i$ is even in the latter case.

By doing the transformations $x=e^{z/\mathbf{p}}$ and $f(z)=y(x)$ to the Hill equation \eqref{Hill Eq} and then the transformation $y=x^{-1/2}u$ to the resulting equation, we obtain
\begin{equation}\label{Bank-laine se3simple hobessel5tran1}
x^{2}u''-\left[\mathbf{p}^2K(x^{\mathbf{p}},x^{-\mathbf{p}})-\frac{1}{4}\right]u=0.
\end{equation}
Then equation \eqref{Bank-laine se3simple hobessel5tran1} has a \emph{Liouvillian} solution of the form $u=x^{1/2}\psi(x)x^{c}e^{\chi(x)}$. See~\cite{mariusmichael2006} for the basic definitions from the differential Galois theory. Denote the field of rational functions by $\mathbb{C}(x)$. Kovacic~\cite{Kovacic1986} has considered the (nonzero) Liouvillian solutions of the linear differential equation $y''=ry$ ($\star$), where $r\in\mathbb{C}(x)$, and proved that precisely four cases can occur:
\begin{itemize}
  \item [(I)] Equation ($\star$) has a solution of the form $e^{\int \omega dx}$ where $\omega\in \mathbb{C}(x)$.
  \item [(II)] Equation ($\star$) has a solution of the form $e^{\int\omega dx}$ where $\omega$ is algebraic over $\mathbb{C}(x)$ of degree~2, and case~(I) does not occur.
  \item [(III)] All solutions of equation ($\star$) are algebraic over $\mathbb{C}(x)$ and cases~(I) and~(II) do not hold.
  \item [(IV)] Equation ($\star$) has no Liouvillian solutions.
\end{itemize}
We shall call \emph{Kovacic's cases} (I), (II), (III) and (IV). These four cases are mutually exclusive and exhaustive. In \cite[Section~2]{Kovacic1986}, Kovacic provided some necessary conditions for each of these four cases to occur and then developed algorithms to find the Liouvillian solutions of equation ($\star$). When $\mathbf{k}=4$ and $\mathbf{l}=0$, Chiang and Yu~\cite[Theorem~3.1]{ChiangYu2019} used Kovacic's algorithms to prove that there is a full correspondence between the non-oscillatory solutions of equation \eqref{Hill Eq} and the Liouvillian solutions of equation \eqref{Bank-laine se3simple hobessel5tran1}.

Here we complete the proof of \cite[Theorem~3.1]{ChiangYu2019}. By doing the transformations $x=e^{z}$ and $f(z)=y(x)$ to the Hill equation \eqref{Hill Eq} and then the transformation $y=x^{-1/2}u$ to the resulting equation, we obtain
\begin{equation}\label{Hill equation-equi}
x^{2}u''-\left[K(x,x^{-1})-\frac{1}{4}\right]u=0.
\end{equation}
If equation \eqref{Hill Eq} has a nonzero non-oscillatory solution $f$ of the form in \eqref{Bank-laine0-eq-1}, then equation \eqref{Hill equation-equi} has a nonzero Liouvillian solution $u$ of the form
\begin{equation}\label{Hill equation-equi-1}
u(x)=x^{1/2}\psi(x^{1/\mathbf{p}})x^{c}e^{\chi(x^{1/\mathbf{p}})}.
\end{equation}
This corresponds to Kovacic's cases (I) or (II) of equation \eqref{Hill equation-equi}. Conversely, we let $u$ be a nonzero solution of equation \eqref{Hill equation-equi}. Then the analytic continuation $f(z)=e^{-z/2}u(e^z)$ is a nonzero solution of equation~\eqref{Hill Eq}. Recall that $\omega$ is an algebraic function over $\mathbb{C}(x)$ if there are polynomials $r_0,\cdots,r_{n}\in\mathbb{C}(x)$ so that $\omega$ satisfies the irreducible equation $r_n\omega^{n}+r_{n-1}\omega^{n-1}+\cdots+r_0=0$. For Kovacic's case (I) of equation \eqref{Hill equation-equi}, by writing $\omega\in \mathbb{C}(x)$ as a partial fraction expansion in the sum of terms of the form $1/x$, $1/x^m$, $1/(x-\alpha)$, $1/(x-\beta)^n$ and $x^{k}$ with integers $m\geq 2$, $n\geq 2$ and $k\geq 0$, we may write $e^{\int \omega dx}$ as the product of terms of the form $x^{\gamma_1}$, $e^{\gamma_2x^{-m+1}}$, $(x-\alpha)^{\gamma_3}$, $e^{\gamma_4(x-\beta)^{-n+1}}$ and $x^{k+1}$. We see that the term $e^{\gamma_4(x-\beta)^{-n+1}}$ cannot appear and $u$ is of the form in \eqref{Hill equation-equi-1} with $\mathbf{p}=1$. It follows that $f$ is a non-oscillatory solution of equation~\eqref{Hill Eq}. For Kovacic's case (II) of equation \eqref{Hill equation-equi}, if $\omega$ has a nonzero branched point, say $x_0$, then $\int \omega dx$ also has a branched point at $x_0$ and then $f=e^{-z/2}u(e^z)$ cannot be an entire function. This implies that $\omega$ can only have branched point at the origin. Then we may also write the integral $\int \omega dx$ explicitly and conclude that $u$ is of the form in \eqref{Hill equation-equi-1} with $\mathbf{p}=2$. It follows that $f$ is a non-oscillatory solution of equation~\eqref{Hill Eq}. For Kovacic's case (III) of equation \eqref{Hill equation-equi}, it is easy to see that $\sigma(f)<\infty$, a contradiction. Recall that all nonzero solutions of equation~\eqref{Hill Eq} are entire functions. Since Kovacic's cases (I), (II), (III) and (IV) of equation \eqref{Hill equation-equi} are mutually exclusive and exhaustive and also that the two cases $\lambda(f)<\infty$ and $\lambda(f)=\infty$ of equation~\eqref{Hill Eq} are mutually exclusive and exhaustive, we conclude the following

\begin{proposition}\label{prososition0}
The Hill equation \eqref{Hill Eq} has a nonzero non-oscillatory solution $f$ if and only if equation \eqref{Hill equation-equi} has a nonzero Liouvillian solution $u$ in Kovacic's case~\emph{(I)} or~\emph{(II)}. Moreover, in Kovacic's case~\emph{(I)}, $\mathbf{k}$ and $\mathbf{l}$ are both even integers and $u$ has the form in \eqref{Hill equation-equi-1} with $\mathbf{p}=1$; in Kovacic's case~\emph{(II)}, $\mathbf{k}$ is an odd integer and $b_{-1}=\cdots=b_{-\mathbf{l}}=0$ and $u$ has the form in \eqref{Hill equation-equi-1} with $\mathbf{p}=2$.
\end{proposition}

By Proposition~\ref{prososition0}, solving the non-oscillatory solutions of the Hill equation~\eqref{Hill Eq} is equivalent to solving the Liouvillian solutions of equation~\eqref{Hill equation-equi}. Thus we may use Kovacic's algorithms to find the non-oscillatory solutions of equation \eqref{Hill Eq}.

The first purpose of this paper is to show that solving the non-oscillatory solutions of the generalize Hill equation \eqref{Hill Eq General} is still equivalent to solving the Liouvillian solutions of equation~\eqref{Hill equation-equi}. To this end, we have to show that $\varrho=1$ when equation \eqref{Hill Eq General} has a nonzero non-oscillatory solution. When $K(e^{z},e^{-\varrho z})=b_{\mathbf{k}}e^{\mathbf{k}z}+b_{-\mathbf{l}}e^{-\varrho\mathbf{l}z}$, this is indeed the case; see Ishizaki and Kazuya~\cite{Ishizakikazuya19971}.

We will actually consider the higher order linear differential equation
\begin{equation}\label{higher order-0}
f^{(n)}+\alpha_{n-1}f^{(n-1)}+\cdots+\alpha_1f'-K(e^{z},e^{-\varrho z})f=0,
\end{equation}
where $n\geq 2$ and $\alpha_1$, $\cdots$, $\alpha_{n-1}$ are constants. Each nonzero solution $f$ of equation \eqref{higher order-0} is an entire function and an elementary application of the lemma on the logarithmic derivative yields that $\sigma(f)=\infty$. We still say that $f$ is a non-oscillatory solution of equation \eqref{higher order-0} if $\lambda(f)<\infty$. We prove the following

\begin{theorem}\label{maintheorem1}
Suppose that equation \eqref{higher order-0} has a nonzero non-oscillatory solution $f$. If $b_{-\mathbf{l}}\neq 0$, then $\varrho=1$. Moreover, when $\varrho=1$, there exist complex constants $c$, a polynomial $\psi$ with simple roots only and a Laurent polynomial $\chi$ such that
\begin{equation}\label{Solution-gen}
f(z)=\psi(e^{z/n})e^{cz}e^{\chi(e^{z/n})}.
\end{equation}
\end{theorem}

By the proof of Theorem~\ref{maintheorem1}, we may write the Lauren polynomial $\chi$ in the form $\chi(z)=\sum_{i=0}^{\mathbf{k}-1}c_{i}z^{\mathbf{k}-i}+\sum_{j=0}^{\mathbf{l}-1}d_{j}z^{-(\mathbf{l}-j)}$. Moreover, $d_0=\cdots=d_{\mathbf{l}}=0$ when $\mathbf{k}$ and $n$ are relatively prime; see Chiang and Wang~\cite[Theorem~2.1]{Chiangwang1997}.

Actually, Shimomura~\cite[Theorem~2.1]{Shimomura2002} has obtained the solution in \eqref{Solution-gen}. In his paper, Shimomura has provided essential generalizations of the results of Bank and Langley~\cite[Theorem~2]{banklangely1992} and Chiang and Wang~\cite[Theorem~2.1]{Chiangwang1997}.

Now we return to the Hill equation \eqref{Hill Eq}. For certain choice of the integers $\mathbf{k}$ and $\mathbf{l}$ in \eqref{bank-laine0-fu-0}, it is possible to precisely characterize all non-oscillatory solutions of equation \eqref{Hill Eq}. See~\cite{Banklaine1983-1} for the case $\mathbf{k}=1$ and $\mathbf{l}=0$ (see also \cite[Theorem~5.22]{Laine1993}) and~\cite{ChiangIsmail2006,Banklaine1983-1,Zhang2022} for the case $\mathbf{k}=2$ and $\mathbf{l}=0$. See also~\cite[Theorem~3.1]{ChiangYu2019} for the case $\mathbf{k}=4$ and $\mathbf{l}=0$. When $K(e^z,e^{-z})$ contains exactly \emph{two} exponential terms, if $K(e^{z},e^{-z})=b_{\mathbf{k}}e^{\mathbf{k}z}+b_0+b_{-\mathbf{l}}e^{-\mathbf{l}z}$ with $b_{\mathbf{k}}b_{-\mathbf{l}}\neq 0$, then equation \eqref{Hill Eq} cannot have any nonzero non-oscillatory solutions~\cite{banklainelangely1989}. We shall consider the particular Hill equation
\begin{equation}\label{bank-laine0002}
f''-\left(e^{\mathbf{k}z}+b_{\mathbf{s}}e^{\mathbf{s}z}+b_0\right)f=0,
\end{equation}
where $\mathbf{k}>\mathbf{s}\geq 1$ are two integers and $b_{\mathbf{s}}\not=0$. Concerning the zero-free solutions of equation \eqref{bank-laine0002}, the present author \cite[Theorem~2.2]{zhang2021-1} proved the following

\begin{theorem}[see \cite{zhang2021-1}]\label{maintheorem2}
Suppose that equation \eqref{bank-laine0002} has a nonzero zero-free solution $f$. Then $\mathbf{s}/\mathbf{k}=1/2$ or $\mathbf{s}/\mathbf{k}=3/4$. Moreover,
\begin{itemize}
\item [(1)]
if $\mathbf{s}=1$ and $\mathbf{k}=2$, then $f(z)=a_0 e^{c_0e^{z}+cz}$, where $a_0$, $c_0$ and $c$ are constants such that $a_0\not=0$, $c_0^2=1$, $2c_0c+c_0=b_1$ and $c^2=b_0$;

\item [(2)]
if $\mathbf{s}=3$ and $\mathbf{k}=4$, then $f(z)=a_0e^{(c_0/2)e^{2z}+c_1e^{z}+cz}$, where $a_0$, $c_0$, $c_1$ and $c$ are constants such that $a_0\not=0$, $c_0^2=1$, $2c_0c_1=b_3$, $c^2=b_0$ and $c_1^2+(2+2c)c_0=0$.
\end{itemize}
\end{theorem}
Theorem~\ref{maintheorem2} was originally stated in a more general setting on second order linear differential equations in~\cite[Theorem~2.2]{zhang2021-1}. The present author~\cite[Theorem~10]{Zhang2022} further considered the nonzero non-oscillatory solutions of equation \eqref{bank-laine0002} and proved the following

\begin{theorem}[see \cite{Zhang2022}]\label{maintheorem3}
Suppose that equation~\eqref{bank-laine0002} has a nonzero non-oscillatory solution~$f$. Then
\begin{itemize}

\item [(1)]
if $\mathbf{s}=1$ and $\mathbf{k}=2$, then $f(z)=(\sum_{j=0}^ka_je^{jz})e^{c_0e^{z}+cz}$, where $k\geq 0$ is an integer, $c_0$ and $c$ are constants such that $c_0^2=1$, $2c_0(c+k)+c_0=b_1$ and $c^2=b_0$, and $a_0$, $\cdots$, $a_k$ are constants such that $a_0a_k\not=0$ and
\begin{equation*}
\begin{split}
2c_0(k+1-j)a_{j-1}=(2jc+j^2)a_j, \ j=1,\cdots,k;
\end{split}
\end{equation*}

\item [(2)]
if $\mathbf{s}=1$ and $\mathbf{k}=4$, then $f(z)=(\sum_{j=-1}^{k+1}a_je^{jz})e^{(c_0/2)e^{2z}+cz}$, where $k\geq 1$ is an integer, $c_0$ and $c$ are constants such that $c_0^2=1$, $2c+2k+2=0$ and $c^2=b_0$, and $a_{-1}$, $a_0$, $\cdots$, $a_k$, $a_{k+1}$ are constants such that $a_0a_k\not=0$, $a_{-1}=a_{k+1}=0$ and
\begin{equation*}
\begin{split}
2c_0(k-j+2)a_{j-2}=-b_1a_{j-1}+(2jc+j^2)a_j, \ j=1,\cdots,k+1;
\end{split}
\end{equation*}

\item [(3)]
if $\mathbf{s}=3$ and $\mathbf{k}=4$, then $f(z)=(\sum_{j=-1}^{k+1}a_je^{jz})e^{(c_0/2)e^{2z}+c_1e^{z}+cz}$, where $k\geq 0$ is an integer, $c_0$, $c_1$ and $c$ are constants such that $c_0^2=1$, $2c_0c_1=b_3$, $c^2=b_0$ and $c_1^2+(2+2c+2k)c_0=0$, and $a_{-1}$, $a_0$, $\cdots$, $a_k$, $a_{k+1}$ are constants such that $a_0a_k\not=0$, $a_{-1}=a_{k+1}=0$ and
\begin{equation*}
\begin{split}
(2k-2j+4)c_0a_{j-2}=(2c+2j-1)c_1a_{j-1}+(2jc+j^2)a_{j}, \ j=1,\cdots,k+1.
\end{split}
\end{equation*}

\end{itemize}

\end{theorem}

A problem left by Theorems~\ref{maintheorem2} and~\ref{maintheorem3} is whether the solutions in Theorem~\ref{maintheorem3} are unique in some sense; see~\cite[Section~5]{Zhang2022}. See also the previous discussions in~\cite{Ishizaki19970,Heittokangasilt2019}. The second purpose of this paper is to resolve this problem by proving the following

\begin{theorem}\label{maintheorem4}
Suppose that equation~\eqref{bank-laine0002} has a nonzero non-oscillatory solution $f$ of the form in \eqref{Bank-laine0-eq-1} and $\psi(z)=a_0+a_kz^k$ with $a_0a_k\not=0$. Then $\mathbf{s}/\mathbf{k}=1/4$ or $\mathbf{s}/\mathbf{k}=2/4$ or $\mathbf{s}/\mathbf{k}=3/4$.
\end{theorem}

Even if we assume that the polynomial $\psi$ in \eqref{Bank-laine0-eq-1} has the form $\psi(z)=a_0+a_{s}z^{s}+a_kz^k$ with $a_0a_sa_k\not=0$, the calculations to determine the constants $a_0$, $a_0$ and $a_k$ will become much more complicated. Nonetheless, Theorem~\ref{maintheorem4} suggests that the existence of a non-oscillatory solution $f$ of \eqref{bank-laine0002} with relatively few zeros is sufficient to single out the three possibilities of equation~\eqref{bank-laine0002}.

The remainder of this paper is structured as follows. In Section~\ref{se2: Tumura--Clunie type differential equation}, we are devoted to proving Theorem~\ref{maintheorem1}. In the first part, we shall formulate Theorem~\ref{maintheorem5} for a general Tumura--Clunie type differential equation \eqref{EQ1} and use the Nevanlinna theory to prove this theorem there. In the second part, we will use Theorem~\ref{maintheorem5} to prove the assertion that $\varrho=1$ when $b_{-\mathbf{l}}\neq 0$ in Theorem~\ref{maintheorem1}. Even though Shimomura~\cite{Shimomura2002} has obtained the solution in \eqref{Solution-gen}, we still complete the proof of Theorem~\ref{maintheorem1} by providing a method for obtaining the function $\psi(e^{z/n})$ in \eqref{Solution-gen}. This method will be further used to determine the zero-free solutions of the higher order linear differential equation \eqref{higher order-1} in the last section of this paper. In both proofs of Theorems~\ref{maintheorem5} and~\ref{maintheorem1}, the Phragm\'{e}n--Lindel\"{o}f theorem (see \cite[Theorem~7.3]{hollandasb}) will be repeatedly used.

In Section~\ref{se3: Kovacic's algorithms for Hill equation}, we use Kovacic's algorithms to prove Theorem~\ref{maintheorem4}. There are mainly two steps in the proof. For the non-oscillatory solution $f$ in \eqref{Bank-laine0-eq-1}, we denote $\kappa=\psi(e^{z/\mathbf{p}})$, $h=\chi(e^{z/\mathbf{p}})$ and $\kappa_c=\kappa e^{cz}$ and thus write $f=\kappa_ce^{h}$. Then, denoting $g=h'$, we obtain from equation \eqref{Hill Eq} that $g$ satisfies a \emph{Tumura--Clunie type differential equation}
\begin{equation}\label{tumura-clunie0}
g^2+g'+2\frac{\kappa'_c}{\kappa_c}g+\frac{\kappa''_c}{\kappa_c}=K(e^{z},e^{-z}).
\end{equation}
In the first step, from equation~\eqref{tumura-clunie0} we determine the Laurent polynomial $\chi$ and the constant $c$ explicitly. It follows that $\kappa_c$ satisfies the linear differential equation
\begin{equation}\label{tumura-clunie1}
\kappa_c''+2\kappa_c'g+\left[g^2+g'-K(e^{z},e^{-z})\right]\kappa_c=0.
\end{equation}
In the second step, from equation \eqref{tumura-clunie1} we write out explicitly the system of linear equations that the coefficients of $\psi$ should satisfy. When $\psi(z)=a_0+a_kz^k$, if $\mathbf{s}/\mathbf{k}\not=1/4,2/4,3/4$, we show that this system of linear equations has no nonzero solutions. See also~Bank~\cite{Bank1993-1} for an "approximate square-root method" to find the non-oscillatory solutions of the Hill equation~\eqref{Hill Eq}.

Finally, in Section~\ref{se5: concluding remarks} we give some comments on the method in the proof of Theorem~\ref{maintheorem1}. We will point out how to solve the non-oscillatory solutions of more higher order linear differential equations and, in particular, prove Theorem~\ref{maintheorem6} there.

\section{Tumura--Clunie type differential equation and applications}\label{se2: Tumura--Clunie type differential equation}

To prove Theorem~\ref{maintheorem1}, in this section we shall first solve entire solutions of a general Tumura--Clunie type differential equation.

\subsection{Part I: Tumura--Clunie type differential equation}\label{subse2: Tumura--Clunie type differential equation}

We denote by $\mathcal{G}$ the set of all meromorphic functions $\phi$ having order of growth $<1$ and by $\mathcal{H}$ the set of all functions $\psi$ such that $\psi=\varphi^{(k)}/\varphi$, $k\geq 1$, for some nonzero meromorphic function $\varphi$ having finite order of growth, respectively. Let $\mathcal{S}=\mathcal{G} \cup\mathcal{H}$. Then we define a \emph{differential polynomial} $P(z,g)$ in $g$ as
\begin{equation*}
P(z,g)=\sum_{i=1}^{k}w_{i}g^{n_{i,0}}(g')^{n_{i,1}}\cdots(g^{(k_i)})^{n_{i,k_i}},
\end{equation*}
where $n_{i,0},\cdots,n_{i,k_i}\in \mathbb{N}$ and all the (nonzero) coefficient functions $w_1$, $\cdots$, $w_{k}$ are in $\mathcal{S}$. Define the \emph{degree} of $P(z,g)$ in $g$ by $\deg_g(P(z,g))=\max\{l_i=\sum_{j=0}^{k_i}n_{i,j}:i=1,\cdots,k\}$. Then our Tumura--Clunie type differential equation takes the form
\begin{equation}\label{EQ1}
g^n+P(z,g)=E,
\end{equation}
where $P(z,g)$ is a differential polynomial in $g$ of degree $\leq n-1$ and $E$ is an \emph{exponential polynomial} of the form
\begin{equation}\label{EQ1-Exp-pol}
E(z)=\sum_{i=0}^{\mathbf{k}-1}\beta_{\mathbf{k}-i}e^{\mu_{\mathbf{k}-i}z}+\sum_{j=0}^{\mathbf{l}-1}\gamma_{\mathbf{l}-j}e^{-\nu_{\mathbf{l}-j}z},
\end{equation}
where $\beta_{\mathbf{k}}$, $\cdots$, $\beta_1$ are nonzero constants and $\mu_{\mathbf{k}}$, $\cdots$, $\mu_1$ are distinct complex numbers with arguments $\arg(\mu_{\mathbf{k}-i})\in(-\pi/2,\pi/2)$, $i=0,\cdots,\mathbf{k}-1$ and $\gamma_{\mathbf{l}}$, $\cdots$, $\gamma_1$ are constants and $\nu_{\mathbf{l}}$, $\cdots$, $\nu_1$ are distinct complex numbers with arguments $\arg(\nu_{\mathbf{l}-j})\in(-\pi/2,\pi/2)$, $j=0,\cdots,\mathbf{l}-1$. We assume that $\gamma_{\mathbf{l}}$, $\cdots$, $\gamma_1$ are all zero or all nonzero.

For each argument $\theta\in(-\pi,\pi]$, we have $\Re(\mu_{\mathbf{k}-i}z)=|\mu_{\mathbf{k}-i}|r\cos(\theta+\arg(\mu_{\mathbf{k}-i}))$, $i=0,\cdots,\mathbf{k}-1$ and $\Re(\nu_{\mathbf{l}-j}z)=|\nu_{\mathbf{l}-j}|r\cos(\theta+\arg(\nu_{\mathbf{l}-j}))$, $j=0,\cdots,\mathbf{l}-1$. The Phragm\'{e}n--Lindel\"{o}f indicator $h_{E}$ for the exponential polynomial $E$ is defined as
\begin{equation*}
\begin{split}
h_{E}(\theta)=\max\left\{\Re(\mu_{\mathbf{k}}e^{\mathbf{i}\theta}),\cdots,\Re(\mu_{1}e^{\mathbf{i}\theta})\right\}
\end{split}
\end{equation*}
when $\gamma_{\mathbf{l}}$, $\cdots$, $\gamma_1$ are all zero, or
\begin{equation*}
\begin{split}
h_{E}(\theta)=\max\left\{\Re(\mu_{\mathbf{k}}e^{\mathbf{i}\theta}),\cdots,\Re(\mu_{1}e^{\mathbf{i}\theta}),-\Re(\mu_{\mathbf{l}}e^{\mathbf{i}\theta}),\cdots,-\Re(\nu_{1}e^{\mathbf{i}\theta})\right\}
\end{split}
\end{equation*}
when $\gamma_{\mathbf{l}}$, $\cdots$, $\gamma_1$ are all nonzero. See \cite{Levin1980}. It is well-known that most zeros of $E$ are located around finitely many \emph{critical lines} of $E$. Except for the critical lines, we have
\begin{equation*}
\begin{split}
\log |E(z)|=h_{E}(\theta)r[1+o(1)]
\end{split}
\end{equation*}
as $z\to\infty$ along a ray $z=re^{\mathbf{i}\theta}$, $\theta\in(-\pi,\pi]$. Thus, by composing $g$ with a suitable rotation, if necessary, we may assume that $|e^{\mu_{\mathbf{k}}z}|>|e^{\mu_{\mathbf{k-1}}z}|\geq \cdots \geq |e^{\mu_1z}|$ on the positive real axis and that $|e^{-\nu_{\mathbf{l}}z}|>|e^{-\nu_{\mathbf{l}-1}z}|\geq \cdots\geq |e^{-\nu_1z}|$ on the negative real axis.

Throughout this section, we always assume the above conditions on equation \eqref{EQ1}. We shall prove the following

\begin{theorem}\label{maintheorem5}
Suppose that equation \eqref{EQ1} has an entire solution $g$. Then there are two integers $\mathbf{m},\mathbf{n}\geq 0$ and complex numbers $\varsigma_0$, $\varsigma_1$, $\cdots$, $\varsigma_{\mathbf{m}}$ and $\tau_0$, $\tau_1$, $\cdots$, $\tau_{\mathbf{n}}$ such that $\varsigma_0=\tau_0=0$ and
\begin{equation}\label{TCsolu}
\begin{split}
g(z)=c(z)+\sum_{i=0}^{\mathbf{m}}\zeta_{i}e^{\frac{\mu_{\mathbf{k}}-\varsigma_i}{n}z}+\sum_{j=0}^{\mathbf{n}}\eta_{j}e^{-\frac{\nu_{\mathbf{l}}-\tau_j}{n}z},
\end{split}
\end{equation}
where $c$ is an entire function having order of growth~$<1$, $\zeta_{0}$, $\cdots$, $\zeta_{\mathbf{m}}$ and $\eta_{0}$, $\cdots$, $\eta_{\mathbf{n}}$ are constants such that $\zeta_{0}^n=\beta_{\mathbf{k}}$ and $\eta_{0}^n=\gamma_{\mathbf{l}}$.

\end{theorem}

By the proof of Theorem~\ref{maintheorem5}, we may suppose that $\arg((\mu_{\mathbf{k}}-\varsigma_i)/n)\in(-\pi/2,\pi/2)$ for all $i=0,\cdots,\mathbf{m}$ and $\arg((\nu_{\mathbf{l}}-\tau_j)/n)\in(-\pi/2,\pi/2)$ for all $j=0,\cdots,\mathbf{n}$. The two series expansions in \eqref{EQ1-algorim-4-cal} and \eqref{EQ1-algorim-5-cal} in the proof allow us to calculate the constants $\zeta_i$ and $\eta_j$, as well as $\varsigma_1$, $\cdots$, $\varsigma_{\mathbf{m}}$ and $\tau_1$, $\cdots$, $\tau_{\mathbf{n}}$, precisely.

When $E$ in \eqref{EQ1-Exp-pol} contains only two exponential terms, to determine the constants $\zeta_i$ and $\eta_j$, as well as $\varsigma_1$, $\cdots$, $\varsigma_{\mathbf{m}}$ and $\tau_1$, $\cdots$, $\tau_{\mathbf{n}}$, we may submit the solution in \eqref{TCsolu} into \eqref{EQ1} and then compare the growth of the terms on both sides of the resulting equation as $z\to\infty$ along the positive and negative rays respectively. See \cite{zhang2021}. For the application in the proof of Theorem~\ref{maintheorem4} in Section~\ref{se3: Kovacic's algorithms for Hill equation}, we have the following

\begin{corollary}\label{corollary1}
Let $E$ in \eqref{EQ1-Exp-pol} be of the form $E(z)=\beta_{2}e^{\mu_{2}z}+\beta_{1}e^{\mu_{1}z}$, where $\mu_{2}$ and $\mu_1$ are real numbers such that $\mu_{2}>\mu_1$. Suppose that equation \eqref{EQ1} has an entire solution $g$. Letting $\mathbf{q}$ be the smallest integer such that $\mathbf{s}_1=\mu_{1}/\mu_{2}\leq [n(\mathbf{q}+1)-1]/[n(\mathbf{q}+1)]$, then
\begin{equation*}
\begin{split}
g(z)=c(z)+\sum_{i=0}^{\mathbf{q}}\zeta_{i}e^{\left[i(\mathbf{s}_1-1)+\frac{1}{n}\right]\mu_{2}z},
\end{split}
\end{equation*}
where $c$ is an entire function having order of growth~$<1$ and $\zeta_{0}$, $\cdots$, $\zeta_{\mathbf{q}}$ are constants such that $\zeta_{0}^{n}=\beta_{2}$ when $\mathbf{q}=0$, and $\zeta_{0}^{n}=\beta_{2}$, $n\zeta_{0}^{n-1}\zeta_{1}=\beta_{1}$ when $\mathbf{q}=1$, and $\zeta_{0}^{n}=\beta_{2}$, $n\zeta_{0}^{n-1}\zeta_{1}=\beta_{1}$, $\sum_{\substack{i_0+\cdots+i_{\mathbf{q}}=n,\\i_1+\cdots+\mathbf{q}i_{\mathbf{q}}=i}}\frac{n!}{i_0!i_1!\cdots i_{\mathbf{q}}!}\zeta_{0}^{i_0}\zeta_{1}^{i_1}\cdots \zeta_{\mathbf{q}}^{i_{\mathbf{q}}}=0$, $i=2,\cdots,\mathbf{q}$ when $\mathbf{q}\geq 2.$

\end{corollary}

To prove Theorem~\ref{maintheorem5}, we may follow the idea in~\cite{zhang2021,zhang2021-1,Zhang2022} to reduce equation \eqref{EQ1} into a non-homogeneous linear differential equation. In the first version of this paper, we have already developed a reduction process for this idea. However, the reduction process is too complicated and here we shall provide an elegant proof for Theorem~\ref{maintheorem5}.

\subsubsection{Preliminaries}\label{subsubse2: preliminaries}

For the coefficients $w_1$, $\cdots$, $w_{k}$ in equation \eqref{EQ1}, we suppose that $w_{i_1}$, $\cdots$, $w_{i_m}$ are in $\mathcal{G}$. Let $\sigma_{0}=\max\{\sigma(w_{i_1}),\cdots,\sigma(w_{i_m})\}$. By assumption, $\sigma_0<1$. From now on we fix a small number $\varepsilon>0$ such that $\sigma_{0}+2\varepsilon<1$. We first have the following

\begin{lemma}\label{orderlemma}
Under the assumptions of Theorem~\ref{maintheorem5}, we have $\sigma(g)=1$.
\end{lemma}

\begin{proof}
For simplicity, denote $P=P(z,g)$. We rewrite $P$ as $P=\sum^{I}_{i=1}w_{i}\tilde{w}_ig^{l_i}$, where $l_i=\sum_{j=0}^{k_i}n_{i,j}$ and $\tilde{w}_i=\prod_{j=1}^{k_i}(g^{(j)}/g)^{n_{i,j}}$.  By the definition of $\mathcal{S}$ and the lemma on the logarithmic derivative, we have $m(r,w_i)=O(r^{\sigma_{0}+\varepsilon})$. Also, we have $m(r,\tilde{w}_i)=o(T(r,g))$, where $r\to\infty$ possibly outside an exceptional set of finite linear measure. Then, by looking at the proof of \cite[Theorem~1.12]{yangy:03}, we deduce from \eqref{EQ1} that
\begin{equation}\label{Preq5}
\begin{split}
m(r,E)=m(g^n+P)=nm(r,g)+o(T(r,g))+O(r^{\sigma_{0}+\varepsilon}),
\end{split}
\end{equation}
where $r\to\infty$ possibly outside an exceptional set of finite linear measure. Note that $E$ in \eqref{EQ1-Exp-pol} is an entire function. By Steinmetz~\cite{Steinmetz1978}, there is a positive real number $\xi$ dependent on $\mu_{\mathbf{k}},\cdots,\mu_1$ and $\nu_{\mathbf{l}},\cdots,\nu_1$ such that
\begin{equation}\label{Preq5-fu}
\begin{split}
m(r,E)=T(r,E)=\xi r[1+o(1)], \ r\to\infty.
\end{split}
\end{equation}
From \eqref{Preq5} and \eqref{Preq5-fu} we get $nT(r,g)=\xi r[1+o(1)]+o(T(r,g))$. Then, for each $\varepsilon>0$, $T(r,g)\leq (1+\varepsilon)\xi r/n$ for all $r$ possibly outside an exceptional set of finite linear measure. By using \cite[lemma~1.1.1]{Laine1993} to remove this exceptional set, if needed, we obtain $\sigma(g)\leq 1$. Thus the error term $o(T(r,g))$ in \eqref{Preq5} can be replaced by $O(\log r)$ without any exceptional set. Then we have $nT(r,g)=\xi r[1+o(1)]$, $r\to\infty$ and the assertion follows.

\end{proof}

By the assumptions of Theorem~\ref{maintheorem5}, we may choose four numbers $\theta_1\in[-\pi/2,0)$, $\theta_2\in(0,\pi/2]$, $\theta_3\in[\pi/2,\pi)$ and $\theta_4\in(\pi,3\pi/2]$ such that $\theta_2=-\theta_1$, $\theta_3=\pi+\theta_1$, $\theta_4=\pi-\theta_1$ and $|e^{\mu_{\mathbf{k}}z}|>|e^{\mu_{\mathbf{k-1}}z}|\geq \cdots\geq |e^{\mu_1z}|$ on rays $z=re^{\mathbf{i}\theta}$ for all $\theta\in(\theta_1,\theta_2)$ and $|e^{-\nu_{\mathbf{l}}z}|>|e^{-\nu_{\mathbf{l}-1}z}|\geq \cdots\geq |e^{-\nu_1z}|$ on ray $z=re^{\mathbf{i}\theta}$ for all $\theta\in(\theta_3,\theta_4)$. Moreover, we may choose $\theta_2$ so that $e^{-\nu_{\mathbf{l}-j}z}\to 0$ as $z\to\infty$ along rays $z=re^{\mathbf{i}\theta}$ for all $\theta\in(\theta_1,\theta_2)$ and all $j=0,\cdots,\mathbf{l}-1$ and $e^{\mu_{\mathbf{k}-i}z}\to 0$ as $z\to\infty$ along rays $z=re^{\mathbf{i}\theta}$ for all $\theta\in(\theta_3,\theta_4)$ and all $i=0,\cdots,\mathbf{k}-1$. Then, for a small $\epsilon>0$, we denote
\begin{equation}\label{sectors-1}
\begin{split}
\mathcal{A}_{1,\epsilon}&=\left\{re^{\mathbf{i}\theta}: \ 0< r<\infty, \ \theta\in(\theta_1+\epsilon,\theta_2-\varepsilon) \right\},\\
\mathcal{A}_{2,\epsilon}&=\left\{re^{\mathbf{i}\theta}: \ 0< r<\infty, \ \theta\in (\theta_3+\epsilon,\theta_4-\epsilon)\right\}.
\end{split}
\end{equation}
Letting $\theta_5=\theta_1+2\pi$, we also denote
\begin{equation}\label{sectors-2}
\begin{split}
\mathcal{A}_{3,\epsilon}&=\left\{re^{\mathbf{i}\theta}: \ 0< r<\infty, \ \theta\in (\theta_2-\epsilon,\theta_3+\epsilon)\right\},\\
\mathcal{A}_{4,\epsilon}&=\left\{re^{\mathbf{i}\theta}: \ 0< r<\infty, \ \theta\in (\theta_4-\epsilon,\theta_5+\epsilon)\right\}.
\end{split}
\end{equation}
To achieve the two exponential sums in \eqref{TCsolu}, we shall analyze the asymptotic behaviors of $g$ as $z\to\infty$ along rays $z=re^{\mathbf{i}\theta}$ in the two sectors $\mathcal{A}_{1,\epsilon}$ and $\mathcal{A}_{2,\epsilon}$.

Due to the appearance of the zeros and poles of $g$ and the coefficient functions $w_1$, $\cdots$, $w_{k}$ in $P(z,g)$, we fail to obtain an asymptotic estimate for $g$ along a ray $z=re^{\mathbf{i}\theta}$ for some $\theta\in(\theta_1,\theta_2)\cup(\theta_3,\theta_4)$. However, the set of all these exceptions is very 'small'. Let $w$ denote $g$ or one of the coefficient functions $w_1$, $\cdots$, $w_{k}$. Recall from~\cite[pp.~84]{Laine1993} that the union of the disks centered at its the zeros and poles whose radii have finite sum form an \emph{$R$--set} in the complex plane. The set of $\theta\in(-\pi,\pi]$ such that the ray $z=re^{\mathbf{i}\theta}$ meets infinitely many disks in the associated $R$-set of $w$ has measure zero. For convenience, in the following we denote by $\mathcal{R}$ the union of all the $R$--sets associated with $g$ and the coefficient functions $w_1$, $\cdots$, $w_{k}$.

When the meromorphic function $w$ is nonconstant, for any $\theta\in(-\pi,\pi]$ such that the ray $z=re^{\mathbf{i}\theta}$ meets finitely many disks in $\mathcal{R}$, by the classical Borel's lemma, we have
\begin{equation}\label{Gro-estimate-1}
\begin{split}
|w(z)|\leq \exp(r^{\sigma(w)+\varepsilon})
\end{split}
\end{equation}
for all $z$ on the ray with large $r$. Moreover, if $w=\varphi^{(k)}/\varphi$ for some meromorphic function $\varphi$ having finite order of growth, then by Gundersen~\cite[Corollary~1]{gundersen:88}, we have
\begin{equation}\label{Gro-estimate-2}
\begin{split}
|w(z)|\leq r^{k[\sigma(\varphi)-1+\varepsilon]}
\end{split}
\end{equation}
for all $z$ on the ray with large $r$. Using the two inequalities \eqref{Gro-estimate-1} and \eqref{Gro-estimate-2}, we may first provide some rough asymptotic estimates for $g$ along rays in $\mathcal{A}_{1,\epsilon}$ and $\mathcal{A}_{2,\epsilon}$. We have the following

\begin{lemma}[\cite{Zhang2022}]\label{growthlemma}
Under the assumptions of Theorem~\ref{maintheorem5}, for any $\theta\in(\theta_1,\theta_2)\cup(\theta_3,\theta_4)$ such that the ray $z=re^{\mathbf{i}\theta}$ meets finitely many disks in $\mathcal{R}$,
\begin{itemize}
\item [(1)]
when $\gamma_{\mathbf{l}}$, $\cdots$, $\gamma_1$ are all zero, if $\theta\in(\theta_1,\theta_2)$, then $g(z)^n=\beta_{\mathbf{k}}e^{\mu_{\mathbf{k}}z}[1+o(1)]$ as $z\to\infty$ along the ray; if $\theta\in(\theta_3,\theta_4)$, then there is an integer $N$ dependent on $\theta$ such that $|g(z)|\leq r^N$ for all $z$ on the ray with large $r$;
\item [(1)]
when $\gamma_{\mathbf{l}}$, $\cdots$, $\gamma_1$ are all nonzero, if $\theta\in(\theta_1,\theta_2)$, then $g(z)^n=\beta_{\mathbf{k}}e^{\mu_{\mathbf{k}} z}[1+o(1)]$ as $z\to\infty$ along the ray; if $\theta \in(\theta_3,\theta_4)$, then $g(z)^n=\gamma_{\mathbf{l}}e^{-\nu_{\mathbf{l}}z}[1+o(1)]$ as $z\to\infty$ along the ray.
\end{itemize}
\end{lemma}

To show that the entire function $c$ in the solution \eqref{TCsolu} has order of growth~$<1$, it is sufficient to have some growth estimates for $g$ along rays in $\mathcal{A}_{1,\epsilon}$ and $\mathcal{A}_{2,\epsilon}$. We have the following

\begin{lemma}\label{smallfunctionlemma}
Let $c$ be an entire function having order of growth $\sigma(c)\leq1$. If for any $\theta\in(\theta_1,\theta_2)\cup (\theta_3,\theta_4)$ such that the ray $z=re^{\mathbf{i}\theta}$ meets finitely many disks in $\mathcal{R}$, we have $|c(z)|\leq \exp(r^{\sigma_{0}+\varepsilon})$ for all $z$ on the ray with large $r$, then $\sigma(c)\leq \sigma_0$.
\end{lemma}

\begin{proof}

Since $c$ is an entire function and~$\sigma(c)\leq1$, we have $|c(z)|\leq \exp(r^{1+\varepsilon})$ for all $z$ with large $r$. By assumption, for a small $\epsilon>0$, we have $|c(z)|\leq \exp(r^{\sigma_{0}+\varepsilon})$ for all $z$ on the two rays $z=re^{\mathbf{i}(\theta_1+\epsilon)}$ and $z=re^{\mathbf{i}(\theta_2-\epsilon)}$ with large $r$. We choose a suitable argument $\vartheta$ so that the analytic function $c(z)\exp(e^{i\vartheta}z^{\sigma_{0}+2\varepsilon})\to 0$ as $z\to\infty$ along the two rays $z=re^{\mathbf{i}(\theta_1+\epsilon)}$ and $z=re^{\mathbf{i}(\theta_2-\epsilon)}$. Note that $|c(z)\exp(e^{\mathbf{i}\vartheta}z^{\sigma_{0}+2\varepsilon})|\leq \exp(r^{1+2\varepsilon})$ for all $z\in \mathcal{A}_{1,\epsilon}$ with large $r$. Then we apply the Phragm\'{e}n--Lindel\"{o}f theorem to the function $c(z)\exp(e^{\mathbf{i}\vartheta}z^{\sigma_{0}+2\varepsilon})$ in the closure $\overline{\mathcal{A}}_{1,\epsilon}$ of $\mathcal{A}_{1,\epsilon}$ and conclude that $c(z)\exp(e^{\mathbf{i}\vartheta}z^{\sigma_{0}+2\varepsilon})\to 0$ uniformly as $z\to\infty$ in $\mathcal{A}_{1,\epsilon}$. Thus $|c(z)|\leq \exp(r^{\sigma_{0}+2\varepsilon})$ for all $z\in \overline{\mathcal{A}}_{1,\epsilon}$ with large $r$. For $\mathcal{A}_{2,\epsilon}$ in \eqref{sectors-1} and $\mathcal{A}_{3,\epsilon}$ and $\mathcal{A}_{4,\epsilon}$ in \eqref{sectors-2}, by similar arguments we also obtain that $|c(z)|\leq \exp(r^{\sigma_{0}+2\varepsilon})$ for all $z\in \overline{\mathcal{A}}_{i,\epsilon}$, $i=2,3,4$, with large $r$. Since $\varepsilon>0$ can be arbitrarily small, we have $\sigma(c)\leq \sigma_0$.

\end{proof}

\subsubsection{Proof of Theorem~\ref{maintheorem5}}\label{subsubse2: proof of maintheorem5}

Denote $\mathbf{w}=P(z,g)/g^{n-1}$ for simplicity. We rewrite equation \eqref{EQ1} as
\begin{equation}\label{EQ1-algorim-3}
g(z)^n\left(1+\frac{\mathbf{w}(z)}{g(z)}\right)=\left(\sum_{i=0}^{\mathbf{k}-1}\beta_{\mathbf{k}-i}e^{\mu_{\mathbf{k}-i}z}\right)\left(1+\frac{\sum_{j=0}^{\mathbf{l}-1}\gamma_{\mathbf{l}-j}e^{-\nu_{\mathbf{l}-j}z}}{\sum_{i=0}^{\mathbf{k}-1}\beta_{\mathbf{k}-i}e^{\mu_{\mathbf{k}-i}z}}\right).
\end{equation}
For any $\theta\in(\theta_1,\theta_2)$ such that the ray $z=re^{\mathbf{i}\theta}$ meets finitely many disks in $\mathcal{R}$, by \eqref{Gro-estimate-1}, \eqref{Gro-estimate-2} and Lemma~\ref{growthlemma}, we deduce that $|\mathbf{w}(z)|\leq \exp(r^{\sigma_{0}+\varepsilon})$ for all $z$ on the ray with large $r$. By Lemma~\ref{growthlemma}, we have $\sum_{i=0}^{\mathbf{k}-1}\beta_{\mathbf{k}-i}e^{\mu_{\mathbf{k}-i}z}=g(z)^n[1+o(1)]$ as $z\to\infty$ along the ray. Note that $\sum_{j=0}^{\mathbf{l}-1}\gamma_{\mathbf{l}-j}e^{-\nu_{\mathbf{l}-j}z}=o(1)$ as $z\to\infty$ along the ray and also that $\mathbf{w}(z)/g(z)=o(1)$ as $z\to\infty$ along the ray. Thus, when $|z|=r$ is large, we may take the \emph{n}-th roots of the meromorphic function on either side of equation \eqref{EQ1-algorim-3} in a small neighbourhood of the ray and fix one branch of them respectively. This gives
\begin{equation*}
g(z)\left(1+\frac{\mathbf{w}(z)}{g(z)}\right)^{1/n}=\left(\sum_{i=0}^{\mathbf{k}-1}\beta_{\mathbf{k}-i}e^{\mu_{\mathbf{k}-i}z}\right)^{1/n}\left[1+\frac{o(1)}{g(z)^n}\right]^{1/n}
\end{equation*}
as $z\to\infty$ along the ray and so, using Taylor's expansion of $(1+x)^{1/n}$ around the origin,
\begin{equation*}
\begin{split}
g(z)+\mathbf{w}_1(z)=\left(\sum_{i=0}^{\mathbf{k}-1}\beta_{\mathbf{k}-i}e^{\mu_{\mathbf{k}-i}z}\right)^{1/n}=\beta_{\mathbf{k}}^{1/n}e^{\frac{\mu_{\mathbf{k}}}{n}z}\left[1+\sum_{i=1}^{\mathbf{k}-1}\frac{\beta_{\mathbf{k}-i}}{\beta_{\mathbf{k}}}e^{(\mu_{\mathbf{k}-i}-\mu_{\mathbf{k}})z}\right]^{1/n},
\end{split}
\end{equation*}
where $\mathbf{w}_1$ satisfies $|\mathbf{w}_1(z)|=O(\exp(r^{\sigma_{0}+\varepsilon}))$ as $z\to\infty$ along the ray. Again, using Taylor's expansion of $(1+x)^{1/n}$ around the origin, we have
\begin{equation}\label{EQ1-algorim-4-cal}
\begin{split}
\beta_{\mathbf{k}}^{1/n}e^{\frac{\mu_{\mathbf{k}}}{n}z}\left[1+\frac{1}{n}\sum_{i=1}^{\mathbf{k}-1}\frac{\beta_{\mathbf{k}-i}}{\beta_{\mathbf{k}}}e^{(\mu_{\mathbf{k}-i}-\mu_{\mathbf{k}})z}+\cdots\right]=\sum_{i=0}^{\infty}\zeta_{i}e^{\frac{\mu_{\mathbf{k}}-\varsigma_i}{n}z}
\end{split}
\end{equation}
for all $z$ on the ray with large $r$, where $\zeta_0,\zeta_1,\cdots$ and $\varsigma_0,\varsigma_1,\cdots$ are constants such that $\zeta_{0}^n=\beta_{\mathbf{k}}$ and $\varsigma_0=0$. Thus, we have
\begin{equation}\label{EQ1-algorim-4}
\begin{split}
g(z)+\mathbf{w}_1(z)=\sum_{i=0}^{\infty}\zeta_{i}e^{\frac{\mu_{\mathbf{k}}-\varsigma_i}{n}z}
\end{split}
\end{equation}
for all $z$ on the ray with large $r$. By composing $g$ with a suitable rotation, if necessary, we may suppose that $\arg((\mu_{\mathbf{k}}-\varsigma_i)/n)\neq \pm\pi/2$ for all $i=0,1,\cdots$. We let $\mathbf{m}\geq 0$ be the biggest integer such that $\arg((\mu_{\mathbf{k}}-\varsigma_{\mathbf{m}})/n)\in (-\pi/2,\pi/2)$.

When $\gamma_{\mathbf{l}}$, $\cdots$, $\gamma_1$ are all zero, for each group of fixed $\zeta_{0}$, $\cdots$, $\zeta_{\mathbf{m}}$, we define the entire function $c$ by
\begin{equation}\label{EQ1-def-small-1}
\begin{split}
c(z)=g-\sum_{i=0}^{\mathbf{m}}\zeta_{i}e^{\frac{\mu_{\mathbf{k}}-\varsigma_i}{n}z}.
\end{split}
\end{equation}
By Lemma~\ref{orderlemma} and \eqref{EQ1-def-small-1}, we see that $\sigma(c)\leq 1$. By \eqref{EQ1-algorim-4-cal} and \eqref{EQ1-algorim-4}, we may suitably choose the number $\theta_2$ with a smaller value so that, for any $\theta\in(\theta_1,\theta_2)$ such that the ray $z=re^{\mathbf{i}\theta}$ meets finitely many disks in $\mathcal{R}$, we have $|c(z)|=O(\exp(r^{\sigma_{0}+\varepsilon}))$ for all $z$ on the ray with large $r$. Moreover, by Lemma~\ref{growthlemma}, for any $\theta\in(\theta_3,\theta_4)$ such that the ray $z=re^{\mathbf{i}\theta}$ meets finitely many disks in $\mathcal{R}$, we obtain from \eqref{EQ1-def-small-1} that $|c(z)|\leq r^{N+1}$ for some integer $N$ and all $z$ on the ray with large $r$. Thus, by Lemma~\ref{smallfunctionlemma} we have $\sigma(c)\leq \sigma_0$.

When $\gamma_{\mathbf{l}}$, $\cdots$, $\gamma_1$ are all nonzero, for any $\theta\in(\theta_3,\theta_4)$ such that the ray $z=re^{\mathbf{i}\theta}$ meets finitely many disks in $\mathcal{R}$, by \eqref{Gro-estimate-1}, \eqref{Gro-estimate-2} and Lemma~\ref{growthlemma}, we deduce that $|\mathbf{w}(z)|\leq \exp(r^{\sigma_{0}+\varepsilon})$ for all $z$ on the ray with large $r$. Moreover, using Taylor's expansion of $(1+x)^{1/n}$ around the origin, we have
\begin{equation}\label{EQ1-algorim-5-cal}
\begin{split}
\gamma_{\mathbf{l}}^{1/n}e^{-\frac{\nu_{\mathbf{l}}}{n}z}\left[1+\frac{1}{n}\sum_{j=1}^{\mathbf{l}-1}\frac{\gamma_{\mathbf{l}-j}}{\gamma_{\mathbf{l}}}e^{-(\nu_{\mathbf{l}-j}-\nu_{\mathbf{l}})z}+\cdots\right]=\sum_{j=0}^{\infty}\eta_{j}e^{-\frac{\nu_{\mathbf{l}}-\tau_j}{n}z}
\end{split}
\end{equation}
for all $z$ on the ray with large $r$, where $\eta_0,\eta_1,\cdots$ and $\tau_0,\tau_1,\cdots$ are constants such that $\eta_{0}^n=\gamma_{\mathbf{l}}$ and $\tau_0=0$. Then, by similar arguments as before, we have
\begin{equation}\label{EQ1-algorim-5}
\begin{split}
g(z)+\mathbf{w}_2(z)=\sum_{j=0}^{\infty}\eta_{j}e^{-\frac{\nu_{\mathbf{l}}-\tau_j}{n}z}
\end{split}
\end{equation}
for all $z$ on the ray with large $r$, where $\mathbf{w}_2$ satisfies $|\mathbf{w}_2(z)|=O(\exp(r^{\sigma_{0}+\varepsilon}))$ as $z\to\infty$ along the ray. By composing $g$ with a suitable rotation, if necessary, we may suppose that $\arg((\nu_{\mathbf{l}}-\tau_j)/n)\neq \pm\pi/2$ for all $j=0,1,\cdots$ and at the same time that $\arg((\mu_{\mathbf{k}}-\varsigma_i)/n)\neq \pm\pi/2$ for all $i=0,1,\cdots$ in \eqref{EQ1-algorim-4}. We let $\mathbf{n}\geq 0$ be the biggest integer such that $\arg((\nu_{\mathbf{l}}-\tau_{\mathbf{n}})/n)\in (-\pi/2,\pi/2)$.

Now, for each fixed group of $\zeta_{0}$, $\cdots$, $\zeta_{\mathbf{m}}$ and each fixed group of $\eta_{0}$, $\cdots$, $\eta_{\mathbf{n}}$, we define the entire function $c$ by
\begin{equation}\label{EQ1-def-small-2}
\begin{split}
c(z)=g(z)-\sum_{i=0}^{\mathbf{m}}\zeta_{i}e^{\frac{\mu_{\mathbf{k}}-\varsigma_i}{n}z}-\sum_{j=0}^{\mathbf{n}}\eta_{j}e^{-\frac{\nu_{\mathbf{l}}-\tau_j}{n}z}.
\end{split}
\end{equation}
By Lemma~\ref{orderlemma} and \eqref{EQ1-def-small-2}, we see that $\sigma(c)\leq 1$. By \eqref{EQ1-algorim-4-cal}, \eqref{EQ1-algorim-4}, \eqref{EQ1-algorim-5-cal} and \eqref{EQ1-algorim-5}, we may suitably choose the number $\theta_2$ with a smaller value so that, for any $\theta\in(\theta_1,\theta_2)\cup (\theta_3,\theta_4)$ such that the ray $z=re^{\mathbf{i}\theta}$ meets finitely many disks in $\mathcal{R}$, we have $|c(z)|=O(\exp(r^{\sigma_{0}+\varepsilon}))$ for all $z$ on the ray with large $r$. Thus, by Lemma~\ref{smallfunctionlemma} we have $\sigma(c)\leq \sigma_{0}$. This completes the proof.

\subsection{Part II: Proof of Theorem~\ref{maintheorem1}}\label{subse2: proof of maintheorem1}

Without loss of generality, we may suppose that $f(0)\neq0$. Since $\sigma(f)=\infty$, we may write $f$ as Hadamard's factorization $f=\kappa e^{\mathbf{h}}$, where $\mathbf{h}$ is a transcendental entire function and $\kappa$ is the canonical product formed by the zeros of $f$ such that $\sigma(\kappa)=\lambda(\kappa)<\infty$. By~\cite[Lemma~8.6]{Laine1993} we have
\begin{equation}\label{bank-laine0-derivative-1}
f^{(k)}=\left\{\kappa^{(k)}+k\mathbf{h}'\kappa^{(k-1)}+\sum_{j=2}^{k}\left[\binom{k}{j}(\mathbf{h}')^{j}+P_{k,j-1}(\mathbf{h}')\right]\kappa^{(k-j)}\right\}e^{\mathbf{h}},\ k=1,\cdots,n,
\end{equation}
where $P_{k,j-1}(\mathbf{h}')$ is a differential polynomial in $\mathbf{h}'$ of degree at most $j-1$ with constant coefficients. It follows that
\begin{equation}\label{bank-laine0-derivative-2}
\frac{f^{(k)}}{f}=\frac{\kappa^{(k)}}{\kappa}+k\mathbf{h}'\frac{\kappa^{(k-1)}}{\kappa}+\sum_{j=2}^{k}\left[\binom{k}{j}(\mathbf{h}')^{j}+P_{k,j-1}(\mathbf{h}')\right]\frac{\kappa^{(k-j)}}{\kappa}, \ k=1,\cdots,n.
\end{equation}
Note that $P_{k,k-1}(\mathbf{h}')=0$. By submitting the solution $f=\kappa e^{\mathbf{h}}$ into equation~\eqref{higher order-0} together with \eqref{bank-laine0-derivative-2}, we get a Tumura--Clunie type differential equation
\begin{equation}\label{bank-laine0-tumura-clunie}
(\mathbf{h}')^n+P(z,\mathbf{h}')=K(e^{z},e^{-\varrho z}),
\end{equation}
where $P(z,\mathbf{h}')$ is a differential polynomial in $\mathbf{h}'$ of degree~$n-1$. By the proof of Theorem~\ref{maintheorem5} together with slight modifications of the arguments in Lemmas~\ref{orderlemma}--\ref{smallfunctionlemma}, we see that
\begin{equation}\label{bank-laine0-tumura-clunie-1}
\mathbf{h}'(z)=c(z)+\sum_{i=0}^{\mathbf{k}-1}\zeta_{i}e^{\frac{\mathbf{k}-i}{n}z}+\sum_{j=0}^{\mathbf{l}-1}\eta_{j}e^{-\varrho\frac{\mathbf{l}-j}{n}z},
\end{equation}
where $c$ is now a polynomial and $\zeta_{0},\cdots,\zeta_{\mathbf{k}}$ and $\eta_{0},\cdots,\eta_{\mathbf{l}}$ are constants such that $\zeta_{0}^2=b_{\mathbf{k}}$ and $\eta_{0}^2=b_{-\mathbf{l}}$.

When $b_{-\mathbf{l}}$, $\cdots$, $b_{-1}$ are all nonzero and $\varrho\neq 1$, we consider the case when $\arg(\varrho)=\theta_0\in(0,\pi)$. The case when $\arg(\varrho)=\theta_0\in(-\pi,0)$ is similar and will be omitted. Denote $g=\mathbf{h}'-c$ and denote the first and the second exponential sum in \eqref{bank-laine0-tumura-clunie-1} by $g_1$ and $g_2$ respectively. Then from \eqref{bank-laine0-tumura-clunie} we have
\begin{equation}\label{bank-laine0-tumura-clunie-2}
g_1^n+\sum_{i=1}^{n-1}\binom{n}{i}g_1^{n-i}g_2^i+g_2^n+Q(z,g)=K(e^{z},e^{-\varrho z}),
\end{equation}
where $Q(z,g)$ is a differential polynomial in $g$ of degree~$n-1$. For any $\theta\in (-\pi/2,\pi/2-\theta_0)$ such that the ray $z=re^{\mathbf{i}\theta}$ meets finitely many disks in $\mathcal{R}$, we see that $g_1(z)\to\infty$ and $g_2(z)\to 0$ as $z\to\infty$ along the ray. Moreover, by \eqref{Gro-estimate-2} and Lemma~\ref{growthlemma}, we have $|Q(z,g)|\leq |z|^{N}|g_1(z)^{n-1}|$ for some integer $N$ and all $z$ on the ray with large $r$. Then we obtain from \eqref{bank-laine0-tumura-clunie-2} that the constants $b_{\mathbf{k}}$, $\cdots$, $b_{1}$ should satisfy
\begin{equation}\label{bank-laine0-tumura-clunie-3}
G_1(z):=\sum_{i=0}^{\mathbf{k}-1}b_{\mathbf{k}-i}e^{(\mathbf{k}-i)z}-g_1(z)^n=O\left(g_1(z)^{n-1}\right)
\end{equation}
as $z\to\infty$ along the ray $z=re^{\mathbf{i}\theta}$, $\theta\in(-\pi/2,\pi/2)$. Similarly, for any $\theta\in (\pi/2,3\pi/2-\theta_0)$ such that the ray $z=re^{\mathbf{i}\theta}$ meets finitely many disks in $\mathcal{R}$, we see that $g_1(z)\to0$ and $g_2(z)\to \infty$ as $z\to\infty$ along the ray and also that $|Q(z,g)|\leq |z|^{N}|g_2(z)^{n-1}|$ for some integer $N$ and all $z$ on the ray with large $r$. Then we obtain from \eqref{bank-laine0-tumura-clunie-2} that the constants $b_{-\mathbf{l}}$, $\cdots$, $b_{-1}$ should satisfy
\begin{equation}\label{bank-laine0-tumura-clunie-4}
G_2(z):=\sum_{j=0}^{\mathbf{l}-1}b_{j-\mathbf{l}}e^{-\varrho(\mathbf{l}-j)z}-g_2(z)^n=O\left(g_2(z)^{n-1}\right)
\end{equation}
as $z\to\infty$ along the ray $z=re^{\mathbf{i}\theta}$, $\theta\in(\pi/2-\theta_0,3\pi/2-\theta_0)$. Now, for a small $\varepsilon_1>0$, we choose one $\theta\in (\pi/2-\theta_0,\pi/2-\theta_0+\varepsilon_1)$ so that the ray $z=re^{\mathbf{i}\theta}$ meets finitely many disks in $\mathcal{R}$ and $g_1(z)\to \infty$, $g_2(z)\to \infty$ and $g_2(z)/g_1(z)\to 0$ as $z\to\infty$ along the ray. We rewrite equation \eqref{bank-laine0-tumura-clunie-2} as
\begin{equation}\label{bank-laine0-tumura-clunie-5}
1+\sum_{i=2}^{n-1}\binom{n}{i}\left(\frac{g_2}{g_1}\right)^{i-1}+\frac{Q(z,g)}{g_1^{n-1}g_2}-\frac{G_1+G_2}{g_1^{n-1}g_2}=0.
\end{equation}
Note that $|Q(z,g)|\leq |z|^{N}|g_1(z)^{n-1}|$ for some integer $N$ and all $z$ on this ray with large $r$. Together with \eqref{bank-laine0-tumura-clunie-3} and \eqref{bank-laine0-tumura-clunie-4}, we obtain from equation \eqref{bank-laine0-tumura-clunie-5} that $1+o(1)=0$ as $z\to\infty$ along the ray, which is impossible. Thus we must have $\varrho=1$.

When $\varrho=1$, we have $\theta_1=-\pi/2$, $\theta_2=\theta_3=\pi/2$ and $\theta_4=3\pi/2$ in $\mathcal{A}_{1,\epsilon}$ and $\mathcal{A}_{2,\epsilon}$ in \eqref{sectors-1} for the exponential polynomial in \eqref{higher order-0}. Denote $\kappa_c=\kappa e^{\int c dz}$ and $h=\mathbf{h}-\int c dz$. Then $f=\kappa_c e^{h}$ and we obtain from \eqref{higher order-0} that $\kappa_c$ satisfies the linear differential equation
\begin{equation}\label{second tumura-clunie-1}
\kappa_c^{(n)}+Q_{1}(h')\kappa_c^{(n-1)}+\cdots+Q_{n-1}(h')\kappa_c'-\left[K(e^{z},e^{-z})-Q_{n}(h')\right]\kappa_c=0,
\end{equation}
where $Q_{i}(h')$, $i=0,\cdots,n-1$ are differential polynomials in $h'$ of degree~$i$ with constant coefficients. We write $K(e^{z},e^{-z})-Q_{n}(h')=Q_{n,1}(e^{z/n})+Q_{n,2}(e^{-z/n})$, where $Q_{n,1}$ and $Q_{n,2}$ are two polynomials, and also write $Q_{i}(h')=Q_{i,1}(e^{z/n})+Q_{i,2}(e^{-z/n})$, $i=1,\cdots,n-1$, where $Q_{i,1}$ and $Q_{i,2}$ are two polynomials.

We need to extend the results of Wittich~\cite{Wittich1967} to show that $c$ is constant and $\kappa=\psi(e^{z/n})$ for some polynomial $\psi$ with simple roots only. Rewrite equation \eqref{second tumura-clunie-1} as
\begin{equation}\label{second tumura-clunie-2}
\frac{\kappa_c'}{\kappa_c}=\frac{K(e^z,e^{-z})-Q_{n}(h')}{Q_{n-1}(h')}-\cdots-\frac{Q_{1}(h')}{Q_{n-1}(h')}\frac{\kappa_c^{(n-1)}}{\kappa_c}-\frac{1}{Q_{n-1}(h')}\frac{\kappa_c^{(n)}}{\kappa_c}.
\end{equation}
Note that $\sigma(\kappa_c)<\infty$. For any $\theta\in(-\pi/2,\pi/2)$ such that the ray $z=re^{\mathbf{i}\theta}$ meets finitely many disks in $\mathcal{R}$, we may use \eqref{Gro-estimate-2} to estimate the modulus of the terms on both sides of equation \eqref{second tumura-clunie-2} along the ray and deduce that the degree of $Q_{n,1}$ should be $\leq n-1$. Then we have
\begin{equation}\label{second tumura-clunie-3}
\frac{\kappa_c'(z)}{\kappa_c(z)}=\mathbf{c}_1+\delta_1(z),
\end{equation}
where $\mathbf{c}_1$ is a constant and $\delta_1(z)$ satisfies $|\delta_1(z)|\leq |e^{-z/2n}|$ for all $z$ on the ray with large $r$. By the remark below Theorem~\ref{maintheorem5}, we see that the constant $\mathbf{c}_1$ is only dependent on $b_{\mathbf{k}}$, $\cdots$, $b_{-\mathbf{l}}$ in \eqref{bank-laine0-fu-0}. Let $z_1$ be a point on the ray $z=re^{\mathbf{i}\theta}$ such that $\kappa_c(z_1)\not=0$ and the line from $z_1$ to $z$ meets no disks in $\mathcal{R}$. Fix the principle branch $\ln \kappa_c$ around the ray. By integration along the line from $z_1$ to $z$ on both sides of equation \eqref{second tumura-clunie-3}, we have
\begin{equation*}
\ln \kappa_c(z)-\ln \kappa_c(z_1)=\mathbf{c}_1z+\int_{z}^z\delta_1(t)dt=\mathbf{c}_1z+\int_{z_1}^{\infty}\delta_1(t)dt-\int_{z}^{\infty}\delta_1(t)dt.
\end{equation*}
Note that $\int_{z_1}^{\infty}\delta_1(t)dt$ converges to a constant, say $\mathbf{c}_0$, and $|\int_{z}^{\infty}\delta_1(t)dt|\leq |e^{-z/4n}|$ for all $z$ on the ray with large $r$. Denote $a_1=\kappa_c(z_1)e^{\mathbf{c}_0}$. Then $a_1\not=0$ and we have
$\kappa_c=a_1e^{\mathbf{c}_1z}[1+o(1)]$ as $z\to\infty$ along the ray. Denote $\Psi=\kappa_ce^{-\mathbf{c}_1z}$. By applying the Phragm\'{e}n--Lindel\"{o}f theorem to $\Psi$ in $\overline{\mathcal{A}}_{1,\epsilon}$, we conclude that
\begin{equation*}
\begin{split}
\Psi(z)\to a_1
\end{split}
\end{equation*}
uniformly as $z\to\infty$ in $\mathcal{A}_{1,\epsilon}$. Moreover, now $\Psi$ satisfies the linear differential equation
\begin{equation}\label{second tumura-clunie-8}
\begin{split}
\Psi^{(n)}&+Q_{1}(h'+\mathbf{c}_1)\Psi^{(n-1)}+\cdots\\
&+Q_{n-1}(h'+\mathbf{c}_1)\Psi'-\left[K(e^z,e^{-z})-Q_{n}(h'+\mathbf{c}_1)\right]\Psi=0.
\end{split}
\end{equation}
By similar arguments as above, there is a constant $a_2\neq0$ such that
\begin{equation}\label{second tumura-clunie-8-add}
\begin{split}
\Psi(nz)\to a_2
\end{split}
\end{equation}
uniformly as $z\to\infty$ in $\mathcal{A}_{1,\epsilon}$. Recall that $\kappa_c=\kappa e^{\int c dz}$ and $h=\mathbf{h}-\int c dz$. To complete the proof, below we shall show that $\Psi$ is a polynomial in $e^{-z/n}$. Since $\sigma(\kappa)=\lambda(\kappa)<\infty$, this will yield that $c$ is constant.

We note that the entire function $\Phi(z)=\Psi(z+2n\pi \mathbf{i})$ is also a nonzero solution of equation \eqref{second tumura-clunie-8}. By similar arguments as before, there is a constant $a_{3}\neq 0$ such that $\Phi(z)\to a_{3}$ uniformly as $z\to\infty$ in $\mathcal{A}_{1,\epsilon}$. If $a_1\Phi\not\equiv a_{3}\Psi$, then $a_1\Phi-a_{3}\Psi$ is also a nonzero solution of equation \eqref{second tumura-clunie-8}. However, then there is a constant $a_{4}\neq 0$ such that $a_1\Phi(z)-a_{3}\Psi(z)\to a_{4}$ uniformly as $z\to\infty$ in $\mathcal{A}_{1,\epsilon}$, which is impossible. Thus $a_1\Phi\equiv a_{3}\Psi$. We write $\Psi(nz)=e^{\mathbf{c}z}w(z)$, where $\mathbf{c}$ is a constant and $w$ is an entire periodic function of period $2\pi \mathbf{i}$. Let $x=e^z$. Then $w(z)=u(x)$ for a function $u$ which is analytic in $\mathbb{C}-\{0\}$. Moreover, $u$ has the Laurent series $u(x)=\sum_{i=-\infty}^{\infty}\alpha_ix^i$.

Note that $w(z)=e^{-\mathbf{c}z}\Psi(nz)$. Then we may choose a ray $z=re^{\mathbf{i}\theta}$ in $\mathcal{A}_{1,\epsilon}$ so that $x=e^{z}\to\infty$ as $z\to\infty$ along the ray and $u$ has a finite limit or tends to $\infty$. Thus $\infty$ is a removable singularity or a pole of $u$, i.e., $u(x)=\sum_{i=-\infty}^{m}\alpha_ix^i$ for some integer $m\geq 0$ and $\alpha_{m}\neq 0$. Now $\Psi(nz)=e^{(\mathbf{c}+m)z}(\sum_{i=-\infty}^{m}\alpha_ie^{(i-m)z})$. If $\mathbf{c}+m\neq0$, then for a small $\epsilon>0$, we may choose a ray in $\mathcal{A}_{1,\epsilon}$ so that $\Psi(nz)\to 0$ or $\Psi(nz)\to\infty$ as $z\to\infty$ along the ray, a contradiction to \eqref{second tumura-clunie-8-add}. Thus $\mathbf{c}+m=0$. We may suppose that $\mathbf{c}=m=0$.

When $\gamma_{\mathbf{l}}$, $\cdots$, $\gamma_1$ are all zero, letting $t=e^{-z/n}$, we may write $\Psi(z)=v(t)$ with an entire function $v$. Then we obtain from \eqref{second tumura-clunie-8} that $v$ satisfies the linear differential equation
\begin{equation}\label{second tumura-clunie-9}
t^nv^{(n)}+S_{n-1}(t,t^{-1})v^{(n-1)}+\cdots+S_{1}(t,t^{-1})v'+S_{0}(t,t^{-1})v=0,
\end{equation}
where $S_{i}(t,t^{-1})=S_{i,1}(t)+S_{i,2}(t^{-1})$, $i=0,\cdots,n-1$, $S_{i,1}$ is a polynomial of degree $\leq i$ and $S_{i,2}$ is a polynomial of degree $\leq n-i$ respectively. With a standard application of the Wiman--Valiron theory to $v$ as in \cite[p.~59]{Laine1993}, we deduce from \eqref{second tumura-clunie-9} that the \emph{central index} $\nu(r)=\nu(r,v)$ is bounded as $r\to\infty$. Thus $v$ is a polynomial.

When $\gamma_{\mathbf{l}}$, $\cdots$, $\gamma_1$ are all nonzero, by similar arguments as before we may also show that the degree of $Q_{n,2}$ should be $\leq n-1$ and there are two constants $\mathbf{d}_1$ and $b_1\not=0$ such that
\begin{equation*}
\begin{split}
\Psi(nz)=b_1e^{\mathbf{d}_1z}[1+o(1)]
\end{split}
\end{equation*}
uniformly as $z\to\infty$ in $\mathcal{A}_{2,\epsilon}$. We may choose a ray $z=re^{\mathbf{i}\theta}$ in $\mathcal{A}_{2,\epsilon}$ so that $x=e^{z}\to 0$ as $z\to\infty$ along the ray and $u$ has a finite limit or tends to $\infty$. Thus $0$ is a removable singularity or a pole of $u$, i.e., $\Psi(nz)=\sum_{i=-k}^{0}\alpha_ie^{iz}$ for some integer $k\geq 0$. This completes the proof.

\section{Proving Theorem~\ref{maintheorem4} using Kovacic's algorithms}\label{se3: Kovacic's algorithms for Hill equation}

By doing the change of variable $z\to 2z$, if necessary, we may assume that the integer $\mathbf{k}$ in equation~\eqref{bank-laine0002} is even. Thus in this section we shall assume that $\mathbf{k}$ and $\mathbf{s}$ are relatively prime or $\mathbf{k}/2$ and $\mathbf{s}/2$ are relatively prime and $\mathbf{k}/2$ is odd. Write $\mathbf{k}=2(\mathbf{m}+1)$ for some integer $\mathbf{m}\geq 0$ and $\mathbf{s}=2(\mathbf{m}+1)-s$ for an integer $s\geq 1$.

By assumption, we may write the nonzero non-oscillatory solution $f$ in the form $f=\kappa e^{cz}e^{h}$, where $\kappa(z)=\psi(e^z)$ for a polynomial $\psi$ with simple roots only, $h(z)=\chi(e^z)$ for a (Laurent) polynomial $\chi$. Moreover, $\psi(z)=\sum_{j=0}^ka_jz^{j}$ with $a_1=\cdots=a_{k-1}=0$. We also write $f=\kappa_ce^{h}$, where $\kappa_c=\kappa e^{cz}$. Denoting $g=h'$, we have from equation \eqref{bank-laine0002} that
\begin{equation}\label{tumura-clunie-particular-1}
g^2+g'+2\frac{\kappa_c'}{\kappa_c}g+\frac{\kappa_c''}{\kappa_c}=b_{\mathbf{k}}e^{\mathbf{k}z}+b_{\mathbf{s}}e^{\mathbf{s}z}+b_0.
\end{equation}
Note that $\kappa_c(z)=\sum_{j=0}^ka_je^{(c+j)z}$. We have
\begin{equation}\label{tumura-clunie-particular-2}
\begin{split}
\frac{\kappa_c'}{\kappa_c}=\frac{\sum_{j=0}^k(c+j)a_je^{jz}}{\sum_{j=0}^ka_je^{jz}}
\end{split}
\end{equation}
and
\begin{equation}\label{tumura-clunie-particular-3}
\begin{split}
\frac{\kappa_c''}{\kappa_c}=\frac{\sum_{j=0}^k(c^2+2jc+j^2)a_je^{jz}}{\sum_{j=0}^ka_je^{jz}}.
\end{split}
\end{equation}

\subsection{Step 1: Determining the polynomial $\chi$ and the constant $c$}\label{subse1: Step I of Kovacic's algorithm}

Note that $g$ is a polynomial in $e^z$ such that $g(z)\to 0$ as $z\to-\infty$. By letting $z\to-\infty$ in equation \eqref{tumura-clunie-particular-1} together with the two expressions in \eqref{tumura-clunie-particular-2} and \eqref{tumura-clunie-particular-3}, we find $b_0=c^2$.

Let $\mathbf{q}$ be the smallest integer such that $\mathbf{s}/\mathbf{k}\leq [2(\mathbf{q}+1)-1]/[2(\mathbf{q}+1)]$. Then, by Corollary~\ref{corollary1} we may write $g=\sum_{i=0}^{\mathbf{q}}\zeta_ie^{(\mathbf{m}+1-is)z}$, where $\zeta_0$, $\cdots$, $\zeta_{\mathbf{q}}$ are constants such that $\zeta_0^2=1$ when $\mathbf{q}=0$ and $\zeta_0^2=1$, $2\zeta_0\zeta_1=b_{\mathbf{s}}$ when $\mathbf{q}=1$ and $\zeta_0^2=1$, $2\zeta_0\zeta_1=b_{\mathbf{s}}$ and
\begin{equation*}
\begin{split}
A_{i}=\sum_{\substack{i_0+\cdots+i_{\mathbf{q}}=2,\\i_1+\cdots+\mathbf{q}i_{\mathbf{q}}=i}}\frac{2}{i_0!\cdots i_{\mathbf{q}}!}\zeta_0^{i_0}\cdots \zeta_{\mathbf{q}}^{i_{\mathbf{q}}}=0, \ i=2,\cdots,\mathbf{q}
\end{split}
\end{equation*}
where $\mathbf{q}\geq 2$. More precisely, when $\mathbf{q}\geq 2$, we have
\begin{equation*}
\begin{split}
A_{i}=2\zeta_0\zeta_i+2\zeta_1\zeta_{i-1}+\cdots+2\zeta_{(i-2)/2}\zeta_{(i+2)/2}+\zeta_{i/2}^2
\end{split}
\end{equation*}
if $2\leq i\leq \mathbf{q}$ is even, and
\begin{equation*}
\begin{split}
A_{i}=2\zeta_0\zeta_i+2\zeta_1\zeta_{i-1}+\cdots+2\zeta_{(i-1)/2}\zeta_{(i+1)/2}
\end{split}
\end{equation*}
if $2\leq i\leq \mathbf{q}$ is odd. Since $A_i=0$ for $i=2,\cdots,\mathbf{q}$, we then compute $\zeta_i$ and find
\begin{equation}\label{coeff1}
\begin{split}
\zeta_i=\frac{t_i\zeta_1^{i}}{(-2\zeta_0)^{i-1}}, \  i=1,\cdots,\mathbf{q},
\end{split}
\end{equation}
where $t_i$ are positive integers such that
\begin{equation}\label{coeff1-fuaf}
\begin{split}
t_1=t_2=1, \ t_3=2, \ t_4=5, \ t_5=14, \ t_6=42, \ t_7=132, \ t_8=401, \ t_9=1374, \ \cdots.
\end{split}
\end{equation}
In particular, we have
\begin{equation}\label{coeff2faf00}
\begin{split}
t_i=2t_1t_{i-1}+\cdots+2t_{(i-2)/2}t_{(i+2)/2}+t_{i/2}^2
\end{split}
\end{equation}
if $2\leq i\leq \mathbf{q}$ is even, and
\begin{equation}\label{coeff2faf100}
\begin{split}
t_i=2t_1t_{i-1}+\cdots+2t_{(i-3)/2}t_{(i+3)/2}+2t_{(i-1)/2}t_{(i+1)/2}
\end{split}
\end{equation}
if $2\leq i\leq \mathbf{q}$ is odd. We may also write out precise formulas for $A_{\mathbf{q}+i}$ for $i=1,\cdots,\mathbf{q}$. When $\mathbf{q}\geq 2$, we have
\begin{equation}\label{coeff2faf-fu1}
\begin{split}
A_{\mathbf{q}+i}=2\zeta_i\zeta_{\mathbf{q}}+2\zeta_{i+1}\zeta_{\mathbf{q}-1}+\cdots+2\zeta_{(\mathbf{q}+i-2)/2}\zeta_{(\mathbf{q}+i+2)/2}+\zeta_{(\mathbf{q}+i)/2}^2
\end{split}
\end{equation}
if $\mathbf{q}+i$ is even, and
\begin{equation}\label{coeff2faf1-fu2}
\begin{split}
A_{\mathbf{q}+i}=2\zeta_i\zeta_{\mathbf{q}}+2\zeta_{i+1}\zeta_{\mathbf{q}-1}+\cdots+2\zeta_{(\mathbf{q}+i-1)/2}\zeta_{(\mathbf{q}+i+1)/2}
\end{split}
\end{equation}
if $\mathbf{q}+i$ is odd. Then, in both of the two cases where $\mathbf{q}=1$ and $\mathbf{q}\geq 2$, by \eqref{coeff1} we easily obtain
\begin{equation}\label{coeff2}
\begin{split}
A_{\mathbf{q}+i}=\sum_{\substack{i_0+\cdots+i_{\mathbf{q}}=2,\\i_1+\cdots+\mathbf{q}i_{\mathbf{q}}=\mathbf{q}+i}}\frac{2}{i_0!\cdots i_{\mathbf{q}}!}\zeta_0^{i_0}\cdots \zeta_{\mathbf{q}}^{i_{\mathbf{q}}}=(-1)^{\mathbf{q}+i-2}\frac{T_{\mathbf{q}+i}\zeta_1^{\mathbf{q}+i}}{(2\zeta_0)^{\mathbf{q}+i-2}}, \  i=1,\cdots,\mathbf{q},
\end{split}
\end{equation}
where $T_{\mathbf{q}+i}$ are positive integers such that $T_3=2$, $T_4=1$ when $\mathbf{q}=2$, and $T_{4}=5$, $T_{5}=4$, $T_{6}=4$ when $\mathbf{q}=3$, and $T_5=14$, $T_6=14$, $T_7=20$, $T_8=25$ when $\mathbf{q}=4$. Note that $t_i$ and $T_{\mathbf{q}+i}$ all depend only on $\mathbf{q}$. We see that $T_{\mathbf{q}+2}<T_{\mathbf{q}+1}$ when $\mathbf{q}\leq 3$. However, when $\mathbf{q}\geq 4$, we have the following

\begin{lemma}\label{lemma  solution0pre1}
Let $\mathbf{q}\geq 4$ be an integer. For the integers $t_i$ and $T_{\mathbf{q}+i}$, we have
\begin{equation}\label{coqnfoiag}
\begin{split}
\frac{t_{\mathbf{q}}}{t_{\mathbf{q}-1}}\geq 2, \ \frac{t_{\mathbf{q}}}{t_{\mathbf{q}-1}}\leq \frac{\mathbf{q}+2}{2} \ \text{and} \ \frac{T_{\mathbf{q}+2}}{T_{\mathbf{q}+1}}\geq 1, \ \frac{T_{\mathbf{q}+3}}{T_{\mathbf{q}+2}}\leq \frac{\mathbf{q}+7}{4}.
\end{split}
\end{equation}
\end{lemma}

\begin{proof}
We prove the first two inequalities in \eqref{coqnfoiag} by mathematical induction. For the first inequality in \eqref{coqnfoiag}, when $3\leq \mathbf{q}\leq 9$, it is easy to see that the numbers in \eqref{coeff1-fuaf} satisfy $t_{i}/t_{i-1}\geq 2$ for all $i=3\cdots,\mathbf{q}$. When $\mathbf{q}\geq 10$, suppose that $t_{i-1}/t_{i-2}\geq 2$ for all $4\leq i\leq k\leq \mathbf{q}$. If $k$ is even, then $(k+2)/2\geq 3$ and thus by \eqref{coeff2faf00} and \eqref{coeff2faf100} we have
\begin{equation}\label{coeff3hilul-0}
\begin{split}
t_{k}=2t_1t_{k-1}+\cdots+2t_{(k-2)/2}t_{(k+2)/2}+t_{k/2}^2>4t_1t_{k-2}+\cdots+4t_{(k-2)/2}t_{k/2}=2t_{k-1}.
\end{split}
\end{equation}
Similarly, if $k$ is odd, then $(k+1)/2\geq 3$ and thus by \eqref{coeff2faf00} and \eqref{coeff2faf100} we have
\begin{equation}\label{coeff3hilul-00}
\begin{split}
t_{k}=2t_1t_{k-1}+\cdots+2t_{(k-1)/2}t_{(k+1)/2}\geq 4t_1t_{k-2}+\cdots+4t_{(k-1)/2}^2>2t_{k-1}.
\end{split}
\end{equation}
By the two inequalities in \eqref{coeff3hilul-0} and \eqref{coeff3hilul-00}, we have $t_{i}/t_{i-1}\geq 2$ for all $i=3,\cdots,\mathbf{q}$ when $\mathbf{q}\geq 3$. This gives the first inequality in \eqref{coqnfoiag}.

For the second inequality in \eqref{coqnfoiag}, when $3\leq \mathbf{q}\leq 9$, it is easy to see that the numbers in \eqref{coeff1-fuaf} satisfy $t_{i}/t_{i-1}\leq (i+2)/2$ for all $i=3,\cdots,\mathbf{q}$. When $\mathbf{q}\geq 10$, suppose that $t_{i-1}/t_{i-2}\leq (i+1)/2$ for all $4\leq i\leq k\leq \mathbf{q}$. If $k\geq 8$ is even, then $(k+2)/2\geq 5$ and thus by \eqref{coeff2faf00} we have
\begin{equation}\label{coeff3hilul}
\begin{split}
t_{k}&=2t_1t_{k-1}+\cdots+2t_{(k-2)/2}t_{(k+2)/2}+t_{k/2}^2\\
&\leq 2\cdot\frac{k+1}{2}t_1t_{k-2}+\cdots+2\left(\frac{k+6}{4}+\frac{k+4}{8}\right)t_{(k-2)/2}t_{k/2}.
\end{split}
\end{equation}
Since $(k+6)/4+(k+4)/8\leq (k+2)/2$ when $k\geq 8$ is even, we obtain from the inequality in \eqref{coeff3hilul} together with \eqref{coeff2faf100} that $t_{k}/t_{k-1}\leq (k+2)/2$. Similarly, if $k\geq 7$ is odd, then $(k-1)/2\geq 3$ and thus by \eqref{coeff2faf100} we have
\begin{equation}\label{coeff3hilul-1}
\begin{split}
t_{k}&=2t_1t_{k-1}+\cdots+2t_{(k-3)/2}t_{(k+3)/2}+t_{(k-1)/2}t_{(k+1)/2}+t_{(k-1)/2}t_{(k+1)/2}\\
&\leq 2\cdot\frac{k+1}{2}t_1t_{k-2}+\cdots+2\left(\frac{k+7}{4}+\frac{k+1}{8}\right)t_{(k-3)/2}t_{(k+1)/2}+\frac{k+5}{4}t_{(k-1)/2}^2.
\end{split}
\end{equation}
Since $(k+7)/4+(k+1)/8\leq (k+2)/2$ and $(k+5)/4<(k+2)/2$ when $k\geq 7$ is odd, we obtain from the inequality in \eqref{coeff3hilul-1} together with \eqref{coeff2faf00} that $t_{k}/t_{k-1}\leq (k+2)/2$. Thus we have $t_{i}/t_{i-1}\leq (i+2)/2$ for all $i=3,\cdots,\mathbf{q}$ when $\mathbf{q}\geq 3$. This gives the second inequality in \eqref{coqnfoiag}.

For the third inequality in \eqref{coqnfoiag}, we have the expressions in \eqref{coeff2faf-fu1} and \eqref{coeff2faf1-fu2}. If $\mathbf{q}\geq 4$ is even, then $(\mathbf{q}+2)/2\geq 3$ and, together with the first inequality in \eqref{coqnfoiag}, we have
\begin{equation}\label{coeff3hilul-nf}
\begin{split}
T_{\mathbf{q}+2}=2t_2t_{\mathbf{q}}+\cdots+t_{(\mathbf{q}+2)/2}^2\geq 2t_1t_{\mathbf{q}}+\cdots+2t_{\mathbf{q}/2}t_{(\mathbf{q}+2)/2}=T_{\mathbf{q}+1}.
\end{split}
\end{equation}
Similarly, if $\mathbf{q}\geq 4$ is odd, then $(\mathbf{q}+1)/2\geq 3$ and, together with the first inequality in \eqref{coqnfoiag}, we have
\begin{equation}\label{coeff3hilul-nf-1}
\begin{split}
T_{\mathbf{q}+2}&=2t_2t_{\mathbf{q}}+\cdots+2t_{(\mathbf{q}+1)/2}t_{(\mathbf{q}+3)/2}\\
&\geq 2t_1t_{\mathbf{q}}+\cdots+2t_{(\mathbf{q}-1)/2}t_{(\mathbf{q}+3)/2}+2t_{(\mathbf{q}+1)/2}^2>T_{\mathbf{q}+1}.
\end{split}
\end{equation}
Thus, by the two inequalities in \eqref{coeff3hilul-nf} and \eqref{coeff3hilul-nf-1}, we have $T_{\mathbf{q}+2}/T_{\mathbf{q}}\geq 1$ when $\mathbf{q}\geq 4$, which is the third inequality in \eqref{coqnfoiag}.

Finally, for the fourth inequality in \eqref{coqnfoiag}, if $\mathbf{q}\geq 4$ is even, then $(\mathbf{q}+2)/2\geq 3$ and, together with the second inequality in \eqref{coqnfoiag}, we have
\begin{equation}\label{coeff3hilul-nf-3}
\begin{split}
T_{\mathbf{q}+3}&=2t_3t_{\mathbf{q}}+\cdots+2t_{(\mathbf{q}+2)/2}t_{(\mathbf{q}+4)/2}\\
&<2\cdot\frac{5}{2}t_2t_{\mathbf{q}}+\cdots+2\cdot\frac{\mathbf{q}+6}{4}t_{\mathbf{q}/2}t_{(\mathbf{q}+4)/2}< \frac{\mathbf{q}+6}{4}T_{\mathbf{q}+2}.
\end{split}
\end{equation}
Similarly, if $\mathbf{q}\geq 4$ is odd, then $(\mathbf{q}+3)/2\geq 4$ and, together with the second inequality in \eqref{coqnfoiag}, we have
\begin{equation}\label{coeff3hilul-nf-4}
\begin{split}
T_{\mathbf{q}+3}&=2t_3t_{\mathbf{q}}+\cdots+t_{(\mathbf{q}+3)/2}^2\\
&\leq 2\cdot\frac{5}{2}t_2t_{\mathbf{q}}+\cdots+\frac{\mathbf{q}+7}{4}t_{(\mathbf{q}+1)/2}t_{(\mathbf{q}+3)/2}< \frac{\mathbf{q}+7}{4}T_{\mathbf{q}+2}.
\end{split}
\end{equation}
Thus, by the two inequalities in \eqref{coeff3hilul-nf-3} and \eqref{coeff3hilul-nf-4}, we have $T_{\mathbf{q}+3}/T_{\mathbf{q}+1}\leq (\mathbf{q}+7)/4$ when $\mathbf{q}\geq 4$, which is the fourth inequality in \eqref{coqnfoiag}. This completes the proof.

\end{proof}

\subsection{Step 2: Determining the polynomial $\psi$}\label{subse2: Step II of Kovacic's algorithm}

We have determined the function $h$ and the constant $c$ in the expression of $f=\kappa e^{h+cz}$ explicitly. Now we begin to determine the function $\kappa$. Recall that $b_0=c^2$. Together with the two expressions in \eqref{tumura-clunie-particular-2} and \eqref{tumura-clunie-particular-3}, we obtain from \eqref{tumura-clunie-particular-1} that
\begin{equation}\label{tumura-clunie-particular-1-fu}
\left(g^2+g'-b_{\mathbf{k}}e^{\mathbf{k}z}-b_{\mathbf{s}}e^{\mathbf{s}z}\right)\sum_{j=0}^ka_je^{jz}+2g\sum_{j=0}^k(c+j)a_je^{jz}+\sum_{j=0}^k(2c+j)ja_je^{jz}=0.
\end{equation}
We see that the left-hand side of equation \eqref{tumura-clunie-particular-1-fu} is a polynomial in $e^{z}$ of degree $k+\mathbf{m}+1$. We may write
\begin{equation*}
\begin{split}
L_{k+\mathbf{m}+1}e^{(k+\mathbf{m}+1)z}+L_{k+\mathbf{m}}e^{(k+\mathbf{m})z}+\cdots+L_{1}e^{z}=0,
\end{split}
\end{equation*}
where $L_1$, $\cdots$, $L_{k+\mathbf{m}+1}$ are linear combinations of $a_0$ and $a_k$. Thus we have $L_{k+\mathbf{m}+1}=\cdots=L_1=0$, giving the system of linear equations that $a_0$ and $a_k$ should satisfy.

Note that $\mathbf{q}=0$ when $\mathbf{s}/\mathbf{k}\leq 1/2$ and $\mathbf{q}\geq 1$ when $\mathbf{s}/\mathbf{k}>1/2$. When $\mathbf{q}\geq 1$, the inequality $(2\mathbf{q}-1)/(2\mathbf{q})<\mathbf{s}/\mathbf{k}\leq [2(\mathbf{q}+1)-1]/[2(\mathbf{q}+1)]$ implies that $\mathbf{q}s<\mathbf{m}+1\leq (\mathbf{q}+1)s$ and $\mathbf{q}\leq \mathbf{m}$. Recall that $\mathbf{k}$ and $\mathbf{s}$ are relatively prime or $\mathbf{m}+1$ and $\mathbf{s}/2$ are relatively prime and $\mathbf{m}+1$ is odd. If $\mathbf{m}+1<(\mathbf{q}+1)s$, then $\mathbf{s}=2(\mathbf{m}+1)-s<2\mathbf{m}+1$, $\mathbf{q}<\mathbf{m}$ and $s\geq 2$; if $\mathbf{m}+1=(\mathbf{q}+1)s$, then $\mathbf{s}=2\mathbf{m}+1$, $\mathbf{q}=\mathbf{m}$ and $s=1$.

When $\mathbf{q}\geq 1$, denote $t=\mathbf{m}+1-\mathbf{q}s$. If $\mathbf{q}<\mathbf{m}$, then we see that $\mathbf{s}<2\mathbf{m}+1$ and $s\geq t\geq 1$. More precisely, we must have $s\geq t+1$, for otherwise we would have $s=t$ and it follows that $\mathbf{k}=2s(\mathbf{q}+1)$ and $\mathbf{s}=s(2\mathbf{q}+1)$, a contradiction to our assumption. If $\mathbf{q}=\mathbf{m}$, then we see that $\mathbf{s}=2\mathbf{m}+1$ and $s=t=1$.

Below we consider the three cases (I) $\mathbf{s}/\mathbf{k}\leq 1/2$, (II) $\mathbf{s}/\mathbf{k}>1/2$ and $\mathbf{s}<2\mathbf{m}+1$, (III) $\mathbf{s}/\mathbf{k}\geq 3/4$ and $\mathbf{s}=2\mathbf{m}+1$ respectively.

\subsubsection{\bf Case~I: $\mathbf{s}/\mathbf{k}\leq 1/2$}\label{subsubse3: case I}

We consider the case when $\mathbf{s}/\mathbf{k}<1/2$. In this case, we have $\mathbf{q}=0$. By substituting $g=\zeta_0e^{(\mathbf{m}+1)z}$ into \eqref{tumura-clunie-particular-1-fu}, we obtain
\begin{equation}\label{polynomial equation E-1}
\begin{split}
\sum_{j=0}^k(2c+2j+\mathbf{m}+1)\zeta_0a_je^{(\mathbf{m}+1+j)z}-\sum_{j=0}^kb_{\mathbf{s}}a_je^{(j+\mathbf{s})z}+\sum_{j=0}^k(2c+j)ja_je^{jz}=0.
\end{split}
\end{equation}
By looking at the term of degree $k+\mathbf{m}+1$ in the polynomial in $e^z$ in the left-hand side of equation \eqref{polynomial equation E-1} and noting that $a_k\not=0$, we find
\begin{equation}\label{first case-con-1}
\begin{split}
2c+2k+\mathbf{m}+1=0.
\end{split}
\end{equation}
By assumption, $a_1=\cdots=a_{k-1}=0$ when $k\geq 2$. Since $\mathbf{m}+1>\mathbf{s}$, we look at the term of degree $\mathbf{s}$ in the polynomial in $e^z$ in the left-hand side of equation \eqref{polynomial equation E-1} and find $\mathbf{s}\leq k$ and also $b_{\mathbf{s}}a_{0}=(2c+\mathbf{s})\mathbf{s}a_{\mathbf{s}}$. Thus $a_{\mathbf{s}}\not=0$ and, by assumption, we have $k=\mathbf{s}$.

If $\mathbf{s}/\mathbf{k}\neq1/4$, then we have $\mathbf{s}<\mathbf{m}+1<2\mathbf{s}$ or $\mathbf{m}+1>2\mathbf{s}$. We look at the term of degree $\mathbf{m}+1$ in the polynomial in $e^z$ in the left-hand side of equation \eqref{polynomial equation E-1} and find $(2c+\mathbf{m}+1)\zeta_0a_0=0$, which is impossible by \eqref{first case-con-1}. Thus we must have $\mathbf{s}/\mathbf{k}=1/4$.

\subsubsection{\bf Case II: $\mathbf{s}/\mathbf{k}>1/2$ \textbf{and} $\mathbf{s}<2\mathbf{m}+1$ }\label{subsubse3: case II}

Now $\mathbf{s}=2(\mathbf{m}+1)-s$ for some integer $s\geq 2$. In this case, we have $\mathbf{q}\geq 1$. By substituting $g=\sum_{i=0}^{\mathbf{q}}\zeta_ie^{(\mathbf{m}+1-is)z}$ into \eqref{tumura-clunie-particular-1-fu}, we obtain
\begin{equation}\label{polynomial equation E-2}
\begin{split}
\left(\sum_{j=0}^ka_je^{jz}\right)\Omega_1(e^z)+2\left[\sum_{j=0}^k(c+j)a_je^{jz}\right]\Omega_2(e^z)+\sum_{j=0}^k(2c+j)ja_je^{jz}=0,
\end{split}
\end{equation}
where
\begin{equation*}
\begin{split}
\Omega_1(e^z)=\sum_{i=0}^{\mathbf{q}}A_{\mathbf{q}+1+i}e^{[2(\mathbf{m}+1)-(\mathbf{q}+1+i)s]z}+\sum_{i=0}^{\mathbf{q}}(\mathbf{m}+1-is)\zeta_ie^{(\mathbf{m}+1-is)z}
\end{split}
\end{equation*}
and
\begin{equation*}
\begin{split}
\Omega_2(e^z)=\sum_{i=0}^{\mathbf{q}}\zeta_ie^{(\mathbf{m}+1-is)z}.
\end{split}
\end{equation*}
By looking at the term of degree $k+\mathbf{m}+1$ in the polynomial in $e^z$ in the left-hand side of equation \eqref{polynomial equation E-2} and noting that $a_k\not=0$, we find
\begin{equation}\label{second case-con-2}
\begin{split}
2c+2k+\mathbf{m}+1=0.
\end{split}
\end{equation}
We write $\mathbf{m}+1-\mathbf{q}s=t\geq 1$. Note that $1\leq \mathbf{q}<\mathbf{m}$. In this case $\mathbf{m}+1-is$ decreases from $\mathbf{m}+1$ to $t$ by $s$ as $i$ increases from $0$ to $\mathbf{q}$ and $2(\mathbf{m}+1)-(\mathbf{q}+1+i)s=2t+(\mathbf{q}-1-i)s$ decreases from $2t+(\mathbf{q}-1)s$ to $2t$ by $s$ as $i$ increases from $0$ to $\mathbf{q}-1$. Note that $2(\mathbf{m}+1)-(2\mathbf{q}+1)s=2t-s\leq t-1$, but $A_{2\mathbf{q}+1}=0$.

By assumption, $a_1=\cdots=a_{k-1}=0$ when $k\geq 2$. We look at the term of degree $t$ in the polynomial in $e^z$ in the left-hand side of equation \eqref{polynomial equation E-2} and find $t\leq k$ and also
\begin{equation*}
\begin{split}
(2c+t)ta_{t}+[(\mathbf{m}+1-\mathbf{q}s)\zeta_{\mathbf{q}}+2c\zeta_{\mathbf{q}}]a_0=(2c+t)(ta_t+\zeta_{\mathbf{q}}a_0)=0.
\end{split}
\end{equation*}
By \eqref{second case-con-2} we see that $2c+t\not=0$. If $t<k$, then $a_t=0$ and it follows that $\zeta_{\mathbf{q}}a_0=0$, a contradiction. Thus, by assumption, we have $k=t$.

Recall that $s\geq t+1$. By the relations in \eqref{coeff2} and Lemma~\ref{lemma  solution0pre1} we see that $A_{\mathbf{q}+1+i}\not=0$ for $i=0,\cdots,\mathbf{q}-1$. By looking at the terms of degrees $t$ and $2t$ in the polynomial in $e^z$ in the left-hand side of equation \eqref{polynomial equation E-2}, we find
\begin{equation}\label{second case-con-3}
\begin{split}
ta_0\zeta_{\mathbf{q}}+2ca_0\zeta_{\mathbf{q}}+(2c+t)ta_t&=0,\\
a_0C_{2\mathbf{q}}+ta_t\zeta_{\mathbf{q}}+2(c+t)a_t\zeta_{\mathbf{q}}&=0.
\end{split}
\end{equation}
Note that $A_{2\mathbf{q}}=\zeta_{\mathbf{q}}^2$. Since $2c+t\not=0$, we obtain from the two equations in \eqref{second case-con-3} that $2(c+t)=0$, which is impossible by \eqref{second case-con-2}. Thus the case when $\mathbf{s}/\mathbf{k}>1/2$ and $\mathbf{s}<2\mathbf{m}+1$ cannot occur.

\subsubsection{\bf Case III: $\mathbf{s}/\mathbf{k}\geq 3/4$ \textbf{and} $\mathbf{s}=2\mathbf{m}+1$}\label{subsubse3: case III}

We consider the case when $\mathbf{s}/\mathbf{k}>3/4$ and $\mathbf{s}=2\mathbf{m}+1$. Then $\mathbf{q}=\mathbf{m}\geq 2$. By substituting $g=\sum_{i=0}^{\mathbf{q}}\zeta_ie^{(\mathbf{m}+1-i)z}$ into \eqref{tumura-clunie-particular-1-fu}, we obtain
\begin{equation}\label{polynomial equation E-3}
\begin{split}
\left(\sum_{j=0}^ka_je^{jz}\right)\Omega_3(e^z)+2\left[\sum_{j=0}^k(c+j)a_je^{jz}\right]\Omega_4(e^z)+\sum_{j=0}^k(2c+j)ja_je^{jz}=0,
\end{split}
\end{equation}
where
\begin{equation*}
\begin{split}
\Omega_3(e^z)=\sum_{i=0}^{\mathbf{m}}[A_{\mathbf{m}+1+i}+(\mathbf{m}+1-i)\zeta_i]e^{(\mathbf{m}+1-i)z}
\end{split}
\end{equation*}
and
\begin{equation*}
\begin{split}
\Omega_4(e^z)=\sum_{i=0}^{\mathbf{m}}\zeta_ie^{(\mathbf{m}+1-i)z}.
\end{split}
\end{equation*}
By looking at the term of degree $k+\mathbf{m}+1$ in the polynomial in $e^z$ in the left-hand side of equation \eqref{polynomial equation E-3} and noting that $a_k\not=0$, we find
\begin{equation}\label{third case-con-1}
\begin{split}
A_{\mathbf{m}+1}+(\mathbf{m}+2c+2k+1)\zeta_0=0.
\end{split}
\end{equation}
By looking at the term of degree $1$ in the polynomial in $e^z$ in the left-hand side of equation \eqref{polynomial equation E-3}, we find
\begin{equation}\label{realtaoni}
\begin{split}
(2c+1)a_1+[A_{2\mathbf{m}+1}+\zeta_{\mathbf{m}}+2c\zeta_{\mathbf{m}}]a_0=(2c+1)(a_{1}+\zeta_{\mathbf{m}}a_0)=0.
\end{split}
\end{equation}
If $a_1=0$, then we must have $2c+1=0$. In this case, $k\geq 2$. Then by looking at the term of degree $k+\mathbf{m}$ in the polynomial in $e^z$ in the left-hand side of equation \eqref{polynomial equation E-2}, we find
\begin{equation}\label{realtaoni-fha}
\begin{split}
A_{\mathbf{m}+2}+(\mathbf{m}+2c+2k)\zeta_1=0.
\end{split}
\end{equation}
Recall from \eqref{coeff2} the formulas
\begin{equation}\label{coeff2-fu}
\begin{split}
A_{\mathbf{m}+1+i}=(-1)^{\mathbf{m}+i-1}\frac{T_{\mathbf{m}+1+i}\zeta_1^{\mathbf{m}+1+i}}{(2\zeta_0)^{\mathbf{m}+i-1}}, \  i=0,\cdots,\mathbf{m}-1,
\end{split}
\end{equation}
where $T_{\mathbf{m}+1+i}$ are positive integers. From the two equations \eqref{third case-con-1} and \eqref{realtaoni-fha} together with the expressions in \eqref{coeff2-fu}, we have $-T_{\mathbf{m}+2}/2T_{\mathbf{m}+1}=(m+2c+2k)/(m+2c+2k+1)$, which is impossible when $2c+1=0$. Thus, by assumption, we have $k=1$.

Denote $\mathbf{d}=(-1)^{\mathbf{m}}T_{\mathbf{m}+1}\zeta_1^{\mathbf{m}+1}/(2\zeta_0)^{\mathbf{m}}$ for simplicity. Since $a_1\not=0$, together with \eqref{coeff2-fu} with $i=0$ we get from the relation in \eqref{third case-con-1} that
\begin{equation}\label{realtaoni-1}
\begin{split}
2\mathbf{d}=\mathbf{m}+2c+3.
\end{split}
\end{equation}
Recall from \eqref{coeff1} the formulas
\begin{equation}\label{coeff1-fu}
\begin{split}
\zeta_i=(-1)^{i-1}\frac{t_i\zeta_1^{i}}{(2\zeta_0)^{i-1}}, \  i=1,\cdots,\mathbf{m},
\end{split}
\end{equation}
where $t_i$ are positive integers. In addition to the numbers $t_1$, $\cdots$, $t_{\mathbf{m}}$ in \eqref{coeff1-fu}, we let $t_0=-1/2$ so that \eqref{coeff1-fu} holds for $i=0$. Now, by looking at the terms of degrees $2$, $\cdots$, $\mathbf{m}+1$ in the polynomial in $e^z$ in the left-hand side of equation \eqref{polynomial equation E-2}, we find
\begin{equation}\label{realtaoni-jjg}
\begin{split}
[A_{\mathbf{m}+1+i}&+(\mathbf{m}+1-i)\zeta_i+2(1+c)\zeta_{i}]a_1\\
&+[A_{\mathbf{m}+i}+(\mathbf{m}+2-i)\zeta_{i-1}+2c\zeta_{i-1}]a_0=0, \ i=1,\cdots,\mathbf{m}.
\end{split}
\end{equation}
Then, by \eqref{coeff2-fu} and \eqref{coeff1-fu}, we get from the relation in \eqref{realtaoni} that
\begin{equation*}
\begin{split}
\left[a_1+(-1)^{\mathbf{m}-1}\frac{t_{\mathbf{m}}\zeta_1^{\mathbf{m}}}{(2\zeta_0)^{\mathbf{m}-1}}a_0\right](2c+1)=0
\end{split}
\end{equation*}
and from the relations in \eqref{realtaoni-jjg} that
\begin{equation*}
\begin{split}
\frac{t_{i}}{t_{i-1}}\left(-\frac{\zeta_1}{2\zeta_0}\right)&\left[\frac{T_{\mathbf{m}+1+i}}{t_{i}T_{\mathbf{m}+1}}\mathbf{d}+\mathbf{m}-i+1+2(1+c)\right]a_1\\
&+\left[\frac{T_{\mathbf{m}+i}}{t_{i-1}T_{\mathbf{m}+1}}\mathbf{d}+\mathbf{m}-i+2+2c\right]a_0=0, \ i=1,\cdots,\mathbf{m}.
\end{split}
\end{equation*}
Denote $u=-a_1\zeta_1/2 a_0\zeta_0$ for simplicity. Recalling $\mathbf{d}=(-1)^{\mathbf{m}}T_{\mathbf{m}+1}\zeta_1^{\mathbf{m}+1}/(2\zeta_0)^{\mathbf{m}}$, we may rewrite the above two equations as
\begin{equation}\label{realtaoni-2-pre-1}
\begin{split}
\left(u+\frac{t_{\mathbf{m}}}{T_{\mathbf{m}+1}}\mathbf{d}\right)(2c+1)=0
\end{split}
\end{equation}
and
\begin{equation}\label{realtaoni-2hi-pre-1}
\begin{split}
u\frac{t_{i}}{t_{i-1}}&\left[\frac{T_{\mathbf{m}+1+i}}{t_{i}T_{\mathbf{m}+1}}\mathbf{d}+\mathbf{m}-i+1+2(1+c)\right]\\
&+\left[\frac{T_{\mathbf{m}+i}}{t_{i-1}T_{\mathbf{m}+1}}\mathbf{d}+\mathbf{m}-i+2+2c\right]=0, \ i=1,\cdots,\mathbf{m}.
\end{split}
\end{equation}
Recall \eqref{realtaoni-1}. When $\mathbf{m}=2$, recalling $t_0=-1/2$, $t_1=t_2=1$ and $T_3=2$, $T_4=1$, the two equations \eqref{realtaoni-2-pre-1} and \eqref{realtaoni-2hi-pre-1} read as $(2u+\mathbf{d})(2c+1)=0$, $2u(3+2c)+(\mathbf{d}+4+4c)=0$ and $u(\mathbf{d}+8+4c)+2=0$. Note that $u\not=0$. If $2c+1=0$, then $\mathbf{d}=2$ and the last two equations become $4u+4=0$ and $8u+2=0$, which are impossible. If $2c+1\not=0$, then from the three equations we obtain $(2c)^2-13(2c)+2=0$ and $5(2c)^2+30(2c)+101=0$, which are also impossible. When $\mathbf{m}=3$, recalling $t_0=-1/2$, $t_1=t_2=1$, $t_3=2$ and $T_{4}=5$, $T_{5}=4$, $T_{6}=4$, the two equations \eqref{realtaoni-2-pre-1} and \eqref{realtaoni-2hi-pre-1} read as $(5u+2\mathbf{d})(2c+1)=0$, $5u(3+2c)+(2\mathbf{d}+5+5c)=0$, $u(4\mathbf{d}+20+10c)+(4\mathbf{d}+15+10c)=0$ and $u(4\mathbf{d}+25+10c)+10=0$. Then it is easy to check by simple computations that there are no $u$ and $c$ satisfying these equations. Below we consider the case when $\mathbf{m}\geq 4$.

Suppose first that $-2c$ is not an integer or is an integer such that $-2c\leq 0$ or $-2c\geq \mathbf{m}+1$. Then $2c+1\not=0$ and we get from the relation in \eqref{realtaoni} that $a_1+\zeta_{\mathbf{m}}a_0=0$. Now $\mathbf{m}+2-i+2c\neq0$ for all $i=2,\cdots,\mathbf{m}$ in \eqref{realtaoni-jjg}. Recalling $A_{2\mathbf{m}+1}=0$, $A_{2\mathbf{m}}=\zeta_{\mathbf{m}}^2$ and $A_{2\mathbf{m}-1}=2\zeta_{\mathbf{m}-1}\zeta_{\mathbf{m}}$, we get from the relation in \eqref{realtaoni-jjg} for $i=\mathbf{m}$ that $\zeta_{\mathbf{m}-1}=\zeta_{\mathbf{m}}^2$ and then from the relation in \eqref{realtaoni-jjg} for $i=\mathbf{m}-1$ that $\zeta_{\mathbf{m}-2}=\zeta_{\mathbf{m}}^3$. Suppose that $\zeta_{i}=\zeta_{\mathbf{m}}^{\mathbf{m}-i+1}$ for $i=k,\cdots,\mathbf{m}$. Then from the relation in \eqref{realtaoni-jjg} we find
\begin{equation*}
\begin{split}
\left(\mathbf{m}+2-k+2c\right)\zeta_{k-1}=\zeta_{\mathbf{m}}\left[\left(\mathbf{m}+3-k+2c\right)\zeta_{k}\right]+\zeta_{\mathbf{m}}C_{\mathbf{m}+1+k}-C_{\mathbf{m}+k}.
\end{split}
\end{equation*}
It is easy to check that $\zeta_{\mathbf{m}}A_{\mathbf{m}+1+k}-A_{\mathbf{m}+k}=-\zeta_{\mathbf{m}}^{\mathbf{m}-k+2}$ using the formulas for $A_{\mathbf{m}+1+i}$, $i=1,\cdots,\mathbf{m}$ in \eqref{coeff2faf-fu1} and \eqref{coeff2faf1-fu2} and thus we have $\zeta_{k-1}=\zeta_{\mathbf{m}}^{\mathbf{m}-k+2}$. By induction we obtain from the equations in \eqref{realtaoni-jjg} for $i=2,\cdots,\mathbf{m}$ that $\zeta_{i}=\zeta_{\mathbf{m}}^{\mathbf{m}-i+1}$ for $i=1,\cdots,\mathbf{m}$. In particular, when $\mathbf{m}\geq 4$, we have $\zeta_2^2=\zeta_1\zeta_3$, i.e., $t_2^2=t_1t_{3}$, which is impossible by the first several numbers $t_i$ in \eqref{coeff1-fuaf}.

Suppose next that $-2c$ is an integer such that $1\leq -2c\leq \mathbf{m}$ and $-2c> (\mathbf{m}+4)/2$. If $2c+1=0$, then $2\mathbf{d}=\mathbf{m}+2$ by \eqref{realtaoni-1}. Recalling $t_0=-1/2$, $t_1=1$ and $T_{2\mathbf{m}}=t_{\mathbf{m}}^2$, the equations in \eqref{realtaoni-2hi-pre-1} with $i=1$ and $i=\mathbf{m}$ read as
\begin{equation*}
\begin{split}
u\left(\frac{T_{\mathbf{m}+2}}{T_{\mathbf{m}+1}}\mathbf{d}+\mathbf{m}+1\right)+1=0
\end{split}
\end{equation*}
and
\begin{equation*}
\begin{split}
2u\frac{t_{\mathbf{m}}}{t_{\mathbf{m}-1}}+\left(\frac{t_{\mathbf{m}}^2}{t_{\mathbf{m}-1}T_{\mathbf{m}+1}}\mathbf{d}+1\right)=0,
\end{split}
\end{equation*}
respectively. It follows that
\begin{equation}\label{realtaoni-2hi-pre-5}
\begin{split}
2\frac{t_{\mathbf{m}}}{t_{\mathbf{m}-1}}=\left(\frac{T_{\mathbf{m}+2}}{T_{\mathbf{m}+1}}\mathbf{d}+\mathbf{m}+1\right)\left(\frac{t_{\mathbf{m}}}{t_{\mathbf{m}-1}}\frac{t_{\mathbf{m}}}{T_{\mathbf{m}+1}}\mathbf{d}+1\right).
\end{split}
\end{equation}
By the inequalities in \eqref{coqnfoiag} in Lemma~\ref{lemma  solution0pre1}, we see that the left-hand side of equation \eqref{realtaoni-2hi-pre-5} is $\leq \mathbf{m}+3$ while the right-hand side of equation \eqref{realtaoni-2hi-pre-5} is $>(3\mathbf{m}+4)/2$, which is impossible when $\mathbf{m}\geq 4$. Thus $2c+1\not=0$. Since $-2c>(\mathbf{m}+4)/2$, then $\mathbf{m}+2-i+2c\neq0$ for all $i=\mathbf{m}/2,\cdots,\mathbf{m}$ when $\mathbf{m}$ is even or for all $i=(\mathbf{m}+1)/2,\cdots,\mathbf{m}$ when $\mathbf{m}$ is odd in \eqref{realtaoni-jjg}. Then by the previous induction process, we have $\zeta_{i}=\zeta_{\mathbf{m}}^{\mathbf{m}-i+1}$ for $i=(\mathbf{m}-2)/2,\cdots,\mathbf{m}$ when $\mathbf{m}$ is an even integer and $\zeta_{i}=\zeta_{\mathbf{m}}^{\mathbf{m}-i+1}$ for $i=(\mathbf{m}-1)/2,\cdots,\mathbf{m}$ when $\mathbf{m}$ is an odd integer. When $\mathbf{m}$ is an even integer, we have $\zeta_{\mathbf{m}-1}^2=\zeta_{\mathbf{m}}\zeta_{\mathbf{m}-2}$, $\cdots$, $\zeta_{\mathbf{m}/2}^2=\zeta_{(\mathbf{m}+2)/2}\zeta_{(\mathbf{m}-2)/2}$, i.e., $t_{\mathbf{m}-1}^2=t_{\mathbf{m}}t_{\mathbf{m}-2}$, $\cdots$, $t_{\mathbf{m}/2}^2=t_{(\mathbf{m}+2)/2}t_{(\mathbf{m}-2)/2}$. Thus we have
\begin{equation*}
\begin{split}
\frac{t_{\mathbf{m}}}{t_{\mathbf{m}-1}}=\frac{t_{\mathbf{m}-1}}{t_{\mathbf{m}-2}}=\cdots=\frac{t_{(\mathbf{m}+2)/2}}{t_{\mathbf{m}/2}}=\frac{t_{\mathbf{m}/2}}{t_{(\mathbf{m}-2)/2}}.
\end{split}
\end{equation*}
Then by the two equations in \eqref{coeff2faf00} and \eqref{coeff2faf100}, we have
\begin{equation*}
\begin{split}
t_{\mathbf{m}}=2t_1\cdot\frac{t_{\mathbf{m}}}{t_{\mathbf{m}-1}}\cdot t_{\mathbf{m}-2}+\cdots+2t_{(\mathbf{m}-2)/2}\cdot\frac{t_{\mathbf{m}}}{t_{\mathbf{m}-1}}\cdot t_{\mathbf{m}/2}+t_{\mathbf{m}/2}^2=t_{\mathbf{m}}+t_{\mathbf{m}/2}^2,
\end{split}
\end{equation*}
which is impossible. Similarly, when $\mathbf{m}$ is an odd integer, we have
\begin{equation*}
\begin{split}
\frac{t_{\mathbf{m}}}{t_{\mathbf{m}-1}}=\frac{t_{\mathbf{m}-1}}{t_{\mathbf{m}-2}}=\cdots=\frac{t_{(\mathbf{m}+3)/2}}{t_{(\mathbf{m}+1)/2}}=\frac{t_{(\mathbf{m}+1)/2}}{t_{(\mathbf{m}-1)/2}}
\end{split}
\end{equation*}
and by the two equations in \eqref{coeff2faf00} and \eqref{coeff2faf100}, we have
\begin{equation*}
\begin{split}
t_{\mathbf{m}}=2t_1\cdot\frac{t_{\mathbf{m}}}{t_{\mathbf{m}-1}}\cdot t_{\mathbf{m}-2}+\cdots+2t_{(\mathbf{m}-1)/2}\cdot\frac{t_{\mathbf{m}}}{t_{\mathbf{m}-1}}\cdot t_{(\mathbf{m}-1)/2}=t_{\mathbf{m}}+t_{(\mathbf{m}-1)/2}t_{(\mathbf{m}+1)/2},
\end{split}
\end{equation*}
which is also impossible.

Suppose now that $-2c$ is an integer such that $1\leq -2c\leq \mathbf{m}$ and $-2c\leq (\mathbf{m}+4)/2$. Recall equation \eqref{realtaoni-1}. Then we have $2\mathbf{d}=\mathbf{m}+2c+3\geq(\mathbf{m}+2)/2\geq 3$. The equations in \eqref{realtaoni-2hi-pre-1} with $i=1$ and $i=2$ read as
\begin{equation*}
\begin{split}
u\left(\frac{T_{\mathbf{m}+2}}{T_{\mathbf{m}+1}}\mathbf{d}+2\mathbf{d}-1\right)+1=0
\end{split}
\end{equation*}
and
\begin{equation*}
\begin{split}
u\left(\frac{T_{\mathbf{m}+3}}{T_{\mathbf{m}+1}}\mathbf{d}+2\mathbf{d}-2\right)+\left(\frac{T_{\mathbf{m}+2}}{T_{\mathbf{m}+1}}\mathbf{d}+2\mathbf{d}-3\right)=0,
\end{split}
\end{equation*}
respectively. It follows that
\begin{equation}\label{realtaoni-2hbuo-3}
\begin{split}
\frac{T_{\mathbf{m}+3}}{T_{\mathbf{m}+1}}\mathbf{d}+2\mathbf{d}-2&=\left(\frac{T_{\mathbf{m}+2}}{T_{\mathbf{m}+1}}\mathbf{d}+2\mathbf{d}-1\right)\left(\frac{T_{\mathbf{m}+2}}{T_{\mathbf{m}+1}}\mathbf{d}+2\mathbf{d}-3\right)\\
&=\left(\frac{T_{\mathbf{m}+2}}{T_{\mathbf{m}+1}}\right)^2\mathbf{d}^2+\frac{T_{\mathbf{m}+2}}{T_{\mathbf{m}+1}}\mathbf{d}(4\mathbf{d}-4)+(2\mathbf{d}-1)(2\mathbf{d}-3).
\end{split}
\end{equation}
By the last two inequalities in \eqref{coqnfoiag} in Lemma~\ref{lemma  solution0pre1} we see that the left-hand and right-hand sides of equation \eqref{realtaoni-2hbuo-3} satisfy
\begin{equation*}
\begin{split}
\text{the left-hand side}
&\leq \frac{T_{\mathbf{m}+2}}{T_{\mathbf{m}+1}}\left(\frac{\mathbf{m}+7}{4}\right)\mathbf{d}+2\mathbf{d}-2,\\
\text{the right-hand side}
&\geq\frac{T_{\mathbf{m}+2}}{T_{\mathbf{m}+1}}\mathbf{d}(5\mathbf{d}-4)+(2\mathbf{d}-1)(2\mathbf{d}-3).
\end{split}
\end{equation*}
Denote
\begin{equation*}
\begin{split} \Delta=\frac{T_{\mathbf{m}+2}}{T_{\mathbf{m}+1}}\mathbf{d}\left(5\mathbf{d}-4-\frac{\mathbf{m}+7}{4}\right)+(2\mathbf{d}-1)(2\mathbf{d}-3)-(2\mathbf{d}-2).
\end{split}
\end{equation*}
Since $5\mathbf{d}-4-(\mathbf{m}+7)/4\geq(4\mathbf{m}-13)/2\geq 3/2$ and thus by Lemma~\ref{lemma  solution0pre1} we have $\Delta \geq 4\mathbf{d}^2-10\mathbf{d}+13/2=(2\mathbf{d}-5/2)^2+1/4>0$, which implies that equation \eqref{realtaoni-2hbuo-3} is impossible. Thus the case when $\mathbf{s}/\mathbf{k}>3/4$ and $\mathbf{s}=2\mathbf{m}+1$ cannot occur.

Finally, by summarising the results in the above three cases together with Theorem~\ref{maintheorem3}, we conclude that $\mathbf{s}/\mathbf{k}=1/4$ or $\mathbf{s}/\mathbf{k}=2/4$ or $\mathbf{s}/\mathbf{k}=3/4$. This completes the proof.

\section{Concluding remarks}\label{se5: concluding remarks}

We show that if the generalized Hill equation \eqref{Hill Eq General} has a nonzero non-oscillatory solution, then it is just the Hill equation. We also point out that there is a full correspondence between the class of non-oscillatory solutions of the Hill equation~\eqref{Hill Eq} and the class of Liouvillian solutions of equation~\eqref{Hill equation-equi}. Kovacic's algorithms are used to find the non-oscillatory solutions of equation~\eqref{Hill Eq}.

We actually give an explicit representation for the non-oscillatory solutions of the higher order linear differential equation~\eqref{higher order-0}. Here we further consider the higher order linear differential equation
\begin{equation}\label{higher order-1}
f^{(n)}+\alpha_{n-1}f^{(n-1)}+\cdots+\alpha_1f'-(E+\beta_0)f=0,
\end{equation}
where $\alpha_1$, $\cdots$, $\alpha_{n-1}$ and $\beta_0$ are constants and $E$ is the exponential polynomial in~\eqref{EQ1-Exp-pol}. Suppose that equation \eqref{higher order-1} has a nonzero non-oscillatory solution $f$ such that $\lambda(f)<1$. By Hadamard's factorization $f=\kappa e^{\mathbf{h}}$ and similar arguments as in the proof of Theorem~\ref{maintheorem1}, we may show that $\mathbf{h}'$ has the form in \eqref{TCsolu} with $c$ being a polynomial. It follows that the function $\kappa_c=\kappa e^{\int c dz}$ satisfies the linear differential equation
\begin{equation}\label{bank-laine0-fu-gen-final-0}
\kappa_c^{(n)}+E_{1}\kappa_c^{(n-1)}+\cdots+E_{n-1}\kappa_c'-\left(E+\beta_0-E_{n}\right)\kappa_c=0,
\end{equation}
where $E_i=Q_{i}(h')$, $i=0,\cdots,n-1$ are differential polynomials in $h'$ of degree~$i$ with constant coefficients and $h'=\mathbf{h}'-c$. Note that $E_1$, $\cdots$, $E_{n-1}$ and $E-E_n$ are all exponential polynomials. Let the critical lines of $h'$ and $E-E_n$ be $\mathcal{L}_i:z=re^{\mathbf{i}\theta_i}$, $i=1,\cdots,m$ with arguments in $[0,2\pi)$ such that $\theta_1<\theta_2<\cdots<\theta_{m}<\theta_{m+1}=\theta_1+2\pi$. For a small $\epsilon>0$, we denote
\begin{equation*}
\mathcal{B}_{i,\epsilon}=\left\{re^{\mathbf{i}\theta}: \ 0< r<\infty, \ \theta\in(\theta_i+\epsilon,\theta_{i+1}-\varepsilon) \right\}, \ i=1,\cdots,m.
\end{equation*}
We may suppose that the Phragm\'{e}n--Lindel\"{o}f indicator $h_{E_{n-1}}(\theta)>0$ in the sectors $\mathcal{B}_{i,\epsilon}$, $i=1,\cdots,m_0$. By similar arguments as in the proof of Theorem~\ref{maintheorem1}, for $i=1,\cdots, m_0$, there are two constants $\mathbf{c}_i$ and $a_i\neq0$ such that
\begin{equation*}
\begin{split}
\kappa_c(z)=a_{i}e^{\mathbf{c}_iz}[1+o(1)]
\end{split}
\end{equation*}
uniformly as $z\to\infty $ in $\mathcal{B}_{i,\epsilon}$. Note that $\sigma(\kappa)=\lambda(\kappa)<1$. If $m_0=m$, then by the Phragm\'{e}n--Lindel\"{o}f theorem, this easily yields that $\sigma(\kappa_c)\leq 1$ and it follows that $c$ is constant. Thus $\kappa(z)=a_{i}e^{(\mathbf{c}_i-c)z}[1+o(1)]$ uniformly as $z\to\infty $ in $\mathcal{B}_{i,\epsilon}$ for $i=1,\cdots,m$. If $\mathbf{c}_i-c\neq 0$ for some $i$, we claim that $\kappa(z)\to 0$ uniformly as $z\to\infty$ in $\mathcal{B}_{i,\epsilon}$. Otherwise, for a small $\epsilon>0$ we may choose a ray in $\mathcal{B}_{i,\epsilon}$ so that $|\kappa(z)|\geq |a_{i}/2||e^{(\mathbf{c}_i-c)z}|$ for all $z$ on the ray with large $r$, a contradiction to our assumption that $\sigma(\kappa)<1$. Thus, for $i=1,\cdots, m$, either $\kappa(z)\to a_{i}$ uniformly as $z\to\infty$ in $\mathcal{B}_{i,\epsilon}$ or $\kappa(z)\to 0$ uniformly as $z\to\infty$ in $\mathcal{B}_{i,\epsilon}$. By the Phragm\'{e}n--Lindel\"{o}f theorem, we conclude that $\kappa$ is constant and nonzero. If $m_0<m$, then $h_{E_{n-1}}(\theta)<0$ for $i=m_0+1,\cdots,m$. For one $\theta$ such that $h_{E_{n-1}}(\theta)<0$ and the ray $z=re^{\mathbf{i}\theta}$ meets finitely many disks in the associated $R$-set of $\kappa$, we obtain from \eqref{bank-laine0-fu-gen-final-0} that
\begin{equation*}
\begin{split}
\frac{\kappa_c^{(n)}(z)}{\kappa_c(z)}=\beta_0+o(1)
\end{split}
\end{equation*}
as $z\to\infty$ along the ray. Since $\sigma(\kappa)<1$, this together with \eqref{Gro-estimate-2} gives $c^n[1+o(1)]=\beta_0+o(1)$ as $z\to\infty$ along the ray, which implies that $c$ is constant. Moreover, for $i=1,\cdots, m_0$, either $\kappa(z)\to a_{i}$ uniformly as $z\to\infty$ in $\mathcal{B}_{i,\epsilon}$ or $\kappa(z)\to 0$ uniformly as $z\to\infty$ in $\mathcal{B}_{i,\epsilon}$. Since the intersection angle of the two critical lines $L_1$ and $L_{m_0+1}$ is $\leq \pi$, then by the Phragm\'{e}n--Lindel\"{o}f theorem, we conclude that $\kappa$ is constant and nonzero. We have proved the following

\begin{theorem}\label{maintheorem6}
Suppose that equation \eqref{higher order-1} has a nonzero non-oscillatory solution $f$ such that $\lambda(f)<1$. Then there are two integers $\mathbf{m},\mathbf{n}\geq 0$ and complex constants $a_0\neq0$, $c$, $c_{i}$, $d_{j}$, $\varsigma_{i}$ and $\tau_{j}$ such that $\varsigma_0=\tau_0=0$ and
\begin{equation*}
\begin{split}
f(z)=a_0e^{cz}\exp\left(\sum_{i=0}^{\mathbf{m}}c_{i}e^{\frac{\mu_{\mathbf{k}}-\varsigma_i}{n}z}+\sum_{j=0}^{\mathbf{n}}d_{j}e^{-\frac{\nu_{\mathbf{l}}-\tau_j}{n}z}\right).
\end{split}
\end{equation*}
\end{theorem}

When $1\leq \lambda(f)<\infty$, if we could determine the form of the entire function $\kappa$ in Hadamard's factorization $f=\kappa e^{\mathbf{h}}$, then we may develop an algorithm to find the non-oscillatory solutions of equation \eqref{higher order-1}. However, the method in the proof of Theorem~\ref{maintheorem1} is not valid any more in this case. On the other hand, we may extend Theorem~\ref{maintheorem6} by replacing $E$ in~\eqref{EQ1-Exp-pol} with the general exponential polynomial
\begin{equation}\label{EQ1-fu-1}
\begin{split}
\mathbf{E}(z)=\sum_{i=0}^{\mathbf{k}-1}\beta_{\mathbf{k}-i}(z)e^{\mu_{\mathbf{k}-i}z^k}+\sum_{j=0}^{\mathbf{l}-1}\gamma_{\mathbf{l}-j}(z)e^{-\nu_{\mathbf{l}-j}z^k},
\end{split}
\end{equation}
where $k\geq 1$ is an integer, $\beta_{\mathbf{k}}$, $\cdots$, $\beta_{1}$ and $\gamma_{\mathbf{l}}$, $\cdots$, $\gamma_{1}$ are entire functions having order of growth~$<k$. Then we may obtain similar conclusions as in Theorems~\ref{maintheorem5} and~\ref{maintheorem6}. In the second order case, see \cite{zhang2021-1} for the case when $\mathbf{E}$ in \eqref{EQ1-fu-1} contains only two exponential terms. See also~\cite{Heittokangasilt2021} for some results.

In the end of this paper, we point out that equation~\eqref{higher order-0} with $\varrho=1$ is a special case of the more general higher order linear differential equation
\begin{equation}\label{higher order-2}
f^{(n)}+K_{n-1}(e^{z},e^{-z})f^{n-1}+\cdots+K_1(e^{z},e^{-z})f'-K(e^{z},e^{-z})f=0,
\end{equation}
where $K_{i}(x,x^{-1})$, $i=1,\cdots,n-1$ are rational in $x$ and analytic in $\mathbb{C}-\{0\}$. Suppose that equation \eqref{higher order-2} has a nonzero non-oscillatory solution $f$. When $\sigma(f)<\infty$, Steinbart~\cite[Theorem~1]{Steinbart1996} has given a precise representation for $f$. When $\sigma(f)=\infty$, by Hadamard's factorization $f=\kappa e^{\mathbf{h}}$ together with~\eqref{bank-laine0-derivative-1} and \eqref{bank-laine0-derivative-2}, we get from equation \eqref{higher order-2} that
\begin{equation*}
P_{n}(z,\mathbf{h}')+K_{n-1}(e^{z},e^{-z})P_{n-1}(z,\mathbf{h}')+\cdots+K_1(e^{z},e^{-z})P_{1}(z,\mathbf{h}')=K(e^{z},e^{-z}),
\end{equation*}
where $P_{n-i}(z,\mathbf{h}')$, $i=0,1,\cdots,n-1$ are differential polynomials in $\mathbf{h}'$ of degree $n-i$. When $K_{i}(x,x^{-1})$, as a sum of two polynomials in $x$ and $x^{-1}$, has smaller degrees in $x$ and $x^{-1}$ than those of $K(x,x^{-1})$ for all $i=1,\cdots,n-1$, we may still have Lemmas~\ref{orderlemma} and~\ref{growthlemma} and then show that $\mathbf{h}'$ is of the form in \eqref{bank-laine0-tumura-clunie-1}. This will still yield the conclusion in Theorem~\ref{maintheorem1}; see~\cite[Theorem~2.1]{Shimomura2002}. In the general case, it is hoped to show that $\mathbf{h}'$ is rational in $e^{z/n}$ by extending the method in the proof of Theorem~\ref{maintheorem5}.

%\section*{Acknowledgements}

%Acknowledgements and other unnumbered sections can be achieved by using the \verb|\section*| command:

%\begin{verbatim}
%\section*{Acknowledgment}
%\end{verbatim}

%\section*{Back Matter}

%\subsection{References} \label{sec3.18}

%\section*{References}

\end{document}